\documentclass[12pt]{amsart}

\usepackage{amssymb, amsfonts, amsmath, amscd, amsthm}
\usepackage{hyperref, eucal, enumerate, mathrsfs}
\usepackage[all]{xy}

\newtheorem{lemma}[subsection]{Lemma}
\newtheorem{proposition}[subsection]{Proposition}
\newtheorem{theorem}[subsection]{Theorem}
\newtheorem{corollary}[subsection]{Corollary}

\theoremstyle{definition}
\newtheorem{pg}[subsection]{}
\newtheorem{definition}[subsection]{Definition}
\newtheorem{notation}[subsection]{Notation}
\newtheorem{remark}[subsection]{Remark}

\numberwithin{equation}{subsection}

\DeclareMathOperator{\ad}{ad}
\DeclareMathOperator{\alg}{alg}
\DeclareMathOperator{\APerf}{APerf}
\DeclareMathOperator{\CAlg}{CAlg}
\DeclareMathOperator{\cn}{cn}
\DeclareMathOperator{\colim}{colim}
\DeclareMathOperator{\Cpl}{Cpl}
\DeclareMathOperator{\DerDM}{DerDM}
\DeclareMathOperator{\fDerDM}{fDerDM}
\DeclareMathOperator{\fSpDM}{fSpDM}
\DeclareMathOperator{\Fun}{Fun}
\DeclareMathOperator{\id}{id}
\DeclareMathOperator{\iTop}{\infty\mathcal{T}op}
\DeclareMathOperator{\LMod}{LMod}
\DeclareMathOperator{\LSym}{LSym}
\DeclareMathOperator{\Map}{Map}
\DeclareMathOperator{\Mod}{Mod}
\DeclareMathOperator{\op}{op}
\DeclareMathOperator{\Poly}{Poly}
\DeclareMathOperator{\QCoh}{QCoh}
\DeclareMathOperator{\RPres}{\mathcal{P}r^{R}}
\DeclareMathOperator{\sHen}{sHen}
\DeclareMathOperator{\Shv}{Shv}
\DeclareMathOperator{\Sp}{Sp} 
\DeclareMathOperator{\SpDM}{SpDM}
\DeclareMathOperator{\Spec}{Spec}
\DeclareMathOperator{\Spf}{Spf}
\DeclareMathOperator{\SSet}{\mathcal{S}}
\DeclareMathOperator{\Sym}{Sym} 

\DeclareMathOperator{\bbE}{\mathbb{E}}
\DeclareMathOperator{\calC}{\mathcal{C}}
\DeclareMathOperator{\calD}{\mathcal{D}}
\DeclareMathOperator{\calO}{\mathcal{O}}
\DeclareMathOperator{\calP}{\mathcal{P}} 
\DeclareMathOperator{\calX}{\mathcal{X}}
\DeclareMathOperator{\calY}{\mathcal{Y}}
\DeclareMathOperator{\frakX}{\mathfrak{X}}
\DeclareMathOperator{\sfX}{\mathsf{X}}
\DeclareMathOperator{\sfY}{\mathsf{Y}} 
\DeclareMathOperator{\sfZ}{\mathsf{Z}}

\newcommand{\et}{\mathrm{\acute et}}

\newcommand{\Adjoint}[4]{\xymatrix@1{#1:#2 \ar@<.4ex>[r] & #3:#4 \ar@<.4ex>[l]}}
\newcommand{\adjoint}[2]{\xymatrix@1{#1 \ar@<.4ex>[r] & #2 \ar@<.4ex>[l]}}
\newcommand{\Pull}[8]{\xymatrix{#1\ar[r]^-{#5}\ar[d]^-{#6} & #2 \ar[d]^-{#7} \\ #3 \ar[r]^-{#8} & #4}}

\title{Formal Derived Algebraic Geometry}
\author{Chang-Yeon Chough}
\address{Department of Mathematics, Sogang University, Seoul 04107, Republic of Korea}
\address{Center for Nano Materials, Sogang University, Seoul 04107, Republic of Korea}
\email{chough@sogang.ac.kr}


\begin{document}

\begin{abstract}
	We develop some foundations for the theory of formal derived algebraic geometry, which parallel the theory of formal spectral algebraic geometry by Jacob Lurie. For this, we establish a close connection between algebro-geometric objects in the derived and spectral settings. We apply this construction to prove a version of the formal GAGA theorem in the derived setting. 
\end{abstract}

\setcounter{tocdepth}{1} 
\maketitle
\tableofcontents

\section{Introduction}

\begin{pg}
	The basic building blocks of ordinary algebraic geometry are commutative rings. Replacing commutative rings by simplicial commutative rings (that is, commutative ring objects in the category of simplicial sets) and connective $\bbE_\infty$-rings (that is, commutative algebra objects in the symmetric monoidal $\infty$-category of connective spectra; see \cite[p.1201]{lurie2017ha}), the theories of derived and spectral algebraic geometry have been developed, respectively. More informally, we can regard a simplicial commutative ring and a connective $\bbE_\infty$-ring as a topological space equipped with an addition and multiplication which satisfy the axioms of commutative rings on the nose and up to coherent homotopy, respectively. When working over an ordinary commutative ring $R$ containing the field $\mathbb{Q}$ of rational numbers, the forgetful functor $\Theta_R: \CAlg^\Delta_R \rightarrow \CAlg_R^{\cn}$ from the $\infty$-category of simplicial commutative $R$-algebras to the $\infty$-category of connective $\bbE_\infty$-algebras over $R$ is an equivalence; see \cite[25.1.2.2]{lurie2018sag}. However, the functor $\Theta_R$ is generally neither fully faithful nor essentially surjective. 
\end{pg}

\begin{pg}
	The main goal of this paper is to establish a relationship between (formal) derived Deligne-Mumford stacks and (formal) spectral Deligne-Mumford stacks. To articulate this more precisely, let us fix a simplicial commutative ring $R$ and let $R^\circ$ denote the underlying $\bbE_\infty$-ring of $R$ (see \ref{the forgetful functor}). With an eye toward Deligne-Mumford stacks in the derived and spectral settings, we note that the forgetful functor $\Theta_R: \CAlg^\Delta_R \rightarrow \CAlg^{\cn}_{R^\circ}$ of \ref{the forgetful functor for a simplicial commutative ring} from the $\infty$-category of simplicial commutative $R$-algebras of \ref{the infinity-category of simplicial commutative rings} to the $\infty$-category of connective $\bbE_\infty$-algebras over $R^\circ$ of \cite[7.1.3.8]{lurie2017ha} determines a functor $\Xi: \iTop_{\CAlg^\Delta_R}^{\sHen} \rightarrow \iTop_{\CAlg^{\cn}_{R^\circ}}^{\sHen}$, where $\iTop_{\CAlg^\Delta_R}^{\sHen}$ denotes the $\infty$-category whose objects are pairs $(\calX, \calO_{\calX})$, where $\calX$ is an $\infty$-topos and $\calO_{\calX}$ is a strictly Henselian $\CAlg^\Delta_R$-valued sheaf on $\calX$, and whose morphism from $(\calX, \calO_{\calX})$ to $(\calY, \calO_{\calY})$ is a pair $(f_\ast, \alpha)$, where $f_\ast: \calX \rightarrow \calY$ is a geometric morphism of $\infty$-topoi and $\alpha: \calO_{\calY} \rightarrow \calO_{\calX} \circ f^\ast$ is a local morphism of $\CAlg^\Delta_R$-valued sheaves on $\calX$, and $\iTop_{\CAlg^{\cn}_{R^\circ}}^{\sHen}$ is defined similarly (see \ref{strictly Henselian objects and local morphisms in the derived setting}, and \ref{strictly Henselian objects and local morphisms in the spectral setting}). 
\end{pg}

\begin{remark}
	The adjoint functors $(\Xi, \Xi^R)$ of \ref{restriction to the subcategory of strictly Henselian objects and local morphisms} restrict to a pair of adjoint functors $\adjoint{\DerDM_R}{\SpDM_{R^\circ}}$ between the $\infty$-category $\DerDM_R$ of derived Deligne-Mumford stacks over $R$ and the $\infty$-category $\SpDM_{R^\circ}$ of spectral Deligne-Mumford stacks over $R^\circ$ of \ref{the infinity-categories of derived and spectral Deligne-Mumford stacks}; see \ref{the adjunction for derived and spectral Deligne-Mumford stacks}. For a simplicial commutative $R$-algebra $A$ and a connective $\bbE_\infty$-algebra $B$ over $R^\circ$, we have canonical equivalences $\Xi(\Spec A) \simeq \Spec \Theta(A)$ and $\Xi(\Spec B) \simeq \Spec (\Theta^L(B))$, where $\Theta^L$ denotes a left adjoint to the functor $\Theta_R: \CAlg^\Delta_R \rightarrow \CAlg^{\cn}_{R^\circ}$ of \ref{the forgetful functor for a simplicial commutative ring}; see \ref{the forgetful functor is monadic and comonadic} and \ref{preservation of affine objects}. 
\end{remark}

\begin{remark}
	The functor $\Xi: \iTop_{\CAlg^\Delta_R}^{\sHen} \rightarrow \iTop_{\CAlg^{\cn}_{R^\circ}}^{\sHen}$ restricts to a functor $\fDerDM_R \rightarrow \fSpDM_{R^\circ}$ between the $\infty$-category $\DerDM_R$ of formal derived Deligne-Mumford stacks over $R$ and the $\infty$-category $\SpDM_{R^\circ}$ of formal spectral Deligne-Mumford stacks over $R^\circ$ of \ref{the infinity-categories of formal derived and spectral Deligne-Mumford stacks}; see \ref{the functor for formal derived and spectral Deligne-Mumford stacks}. If $A$ is an adic simplicial commutative $R$-algebra with a finitely generated ideal of definition $I \subseteq \pi_0A$, then we have a canonical equivalence $\Xi(\Spf A) \rightarrow \Spf \Theta(A)$, where we regard $\Theta(A)$ as an adic $\bbE_\infty$-algebra over $R^\circ$ with the same ideal of definition as $A$; see \ref{adic derived rings} and \ref{the underlying forma spectra}.
\end{remark}

\begin{pg}
	As an application of our construction, we prove a version of the formal GAGA theorem in derived algebraic geometry. In \cite{lurie2018sag}, Lurie set up the theory of formal spectral algebraic geometry and proved many principal results such as the Grothendieck existence theorem and the formal GAGA theorem. Most of the definitions of the theory can be adapted to the setting of derived algebraic geometry without essential change. However, some results cannot be translated to the derived setting: for example, one of the main ingredients in the proof of the formal GAGA principle in the spectral setting of \cite[8.5.3.1]{lurie2018sag} is Tannaka duality of \cite[9.6.4.2]{lurie2018sag}. Nevertheless, we use the compatibility of formal completions in the derived and spectral settings of \ref{comparison for completions} to establish the following version of the formal GAGA theorem in the setting of derived algebraic geometry:
\end{pg}

\begin{theorem}\label{formal GAGA in DAG}
	Let $R$ be a complete adic simplicial commutative ring with a finitely generated ideal of definition $I \subseteq \pi_0R$ (see \emph{\ref{adic simplicial commutative rings}}), and suppose that $R$ is noetherian. Let $\sfX$ be a derived algebraic space which is proper and locally almost of finite presentation over $R$. Let $\sfY$ be a quasi-separated derived algebraic space. Then the restriction map 
$$
\theta_{\sfX}: \Map_{\DerDM}(\sfX, \sfY) \rightarrow \Map_{\fDerDM}(\sfX^\wedge_I, \sfY),
$$ 
given by precomposition with the canonical map $\sfX^\wedge_I \rightarrow \sfX$, is a homotopy equivalence. Here $\sfX^\wedge_I$ denotes the formal completion of $\sfX$ along the vanishing locus of $I$ (see \emph{\ref{formal completion in DAG}}). 
\end{theorem}

\begin{remark}
	In the statement of \ref{formal GAGA in DAG}, we do not assume that $\sfY$ is locally noetherian (see \ref{locally noetherian derived Deligne-Mumford stacks}). In the special case where $\sfY$ is locally noetherian, \ref{formal GAGA in DAG} follows from \cite[5.1.13]{MR4560539} of Halpern-Leistner and Preygel. 
\end{remark}

\begin{pg}\textbf{Conventions}.
	We will make use of the theory of $\infty$-categories and the theory of spectral algebraic geometry developed in \cite{MR2522659}, \cite{lurie2017ha}, and \cite{lurie2018sag}. 
\end{pg}

\begin{pg}\textbf{Acknowledgements}. 
	This work was supported by Samsung Science and Technology Foundation under Project Number SSTF-BA2302-01 and the National Research Foundation of Korea (NRF) grant funded by the Korean government (MSIT) (No. RS-2024-00441954).
\end{pg}

\section{Comparison of Simplicial Commutative Rings and $\bbE_\infty$-rings}

\begin{pg}
	Our goal in this section is to study the relationship between simplicial commutative rings and connective $\bbE_\infty$-rings. First, we recall some terminology from \cite[25.1.1.1]{lurie2018sag}. 
\end{pg}

\begin{pg}
	Let $\Poly_\mathbb{Z}$ denote the ordinary category whose objects are polynomial rings $\mathbb{Z}[x_1, \ldots, x_n]$ for $n \geq 0$, and whose morphisms are homomorphisms of $\mathbb{Z}$-algebras. We will regard $\Poly_\mathbb{Z}$ as an $\infty$-category by taking its nerve; see \cite[1.1.2.6]{MR2522659}. Let $\SSet$ denote the $\infty$-category of spaces of \cite[1.2.16.1]{MR2522659}. We will denote the full subcategory of $\Fun(\Poly_\mathbb{Z}, \SSet)$ spanned by those functors which preserve finite products by $\CAlg^\Delta$ and refer to it as the \emph{$\infty$-category of simplicial commutative rings}. 
\end{pg}

\begin{notation}\label{the infinity-category of simplicial commutative rings}
	If $R$ is a simplicial commutative ring, we let $\CAlg^\Delta_R$ denote the undercategory $(\CAlg^\Delta)_{R/}$ and refer to its objects as \emph{simplicial commutative $R$-algebras}. If $R$ is a connective $\bbE_\infty$-ring in the sense of \cite[7.1.0.1]{lurie2017ha} (see also \cite[p.1201]{lurie2017ha}), we let $\CAlg^{\cn}_R$ denote the \emph{$\infty$-category of connective $\bbE_\infty$-algebras over $R$}; see \cite[7.1.3.8]{lurie2017ha}. 
\end{notation}

\begin{pg}\label{compact generation}
	Let $R$ be a simplicial commutative ring. Then the $\infty$-category $\CAlg^\Delta_R$ can be characterized as in \cite[25.1.1.2]{lurie2018sag}. For this, let $R[x_1, \ldots, x_n]$ denote the tensor product $R\otimes_{\mathbb{Z}}\mathbb{Z}[x_1, \ldots, x_n]$ in the $\infty$-category $\CAlg^\Delta$ and let $\Poly_R$ denote the full subcategory of $\CAlg^\Delta_R$ spanned by the polynomial rings $R[x_1, \ldots, x_n]$, where $n \geq 0$. Since $\Poly_R \subseteq \CAlg^\Delta_R$ is closed under finite coproducts, it follows from \cite[5.5.8.25]{MR2522659} that the canonical map $\calP_{\Sigma}(\Poly_R) \rightarrow \CAlg^\Delta_R$ is an equivalence of $\infty$-categories, where $\calP_{\Sigma}(\Poly_R)$ denotes the full subcategory of $\Fun(\Poly_R^{\op}, \SSet)$ spanned by those functors which preserve finite products. Applying \cite[5.5.8.10]{MR2522659}, we conclude that $\CAlg^\Delta_R$ is a compactly generated $\infty$-category in the sense of \cite[5.5.7.1]{MR2522659}. Note that the collection of polynomial algebras $R[x_1, \ldots, x_n]$ forms a set of compact projective generators for $\CAlg^\Delta_R$ in the sense of \cite[5.5.8.23]{MR2522659}. 
\end{pg}

\begin{pg}\label{the forgetful functor}
	Let us regard $\mathbb{Z}$ as a discrete $\bbE_\infty$-ring; see \cite[7.1.0.3]{lurie2017ha}. Then \cite[25.1.2.1]{lurie2018sag} supplies an essentially unique functor $\Theta_\mathbb{Z}: \CAlg^\Delta \rightarrow \CAlg^{\cn}_\mathbb{Z}$ which preserves small sifted colimits and is the identity on $\Poly_\mathbb{Z}$. For each object $A \in \CAlg^\Delta$, we will denote the image of $A$ under the functor $\Theta_\mathbb{Z}$ by $A^\circ$ and refer to it as the \emph{underlying connective $\bbE_\infty$-algebra of $A$ over $\mathbb{Z}$}. 
\end{pg}

\begin{definition}\label{the forgetful functor for a simplicial commutative ring}
	Let $R$ be a simplicial commutative ring. We let $\Theta_R: \CAlg^\Delta_R \rightarrow \CAlg^{\cn}_{R^\circ}$ denote the functor determined by $\Theta_\mathbb{Z}$ (see \ref{the infinity-category of simplicial commutative rings}). The functor $\Theta_R$ carries a simplicial commutative $R$-algebra $A$ to $\Theta_\mathbb{Z}(A)=A^\circ$, which we will refer to as the \emph{underlying connective $\bbE_\infty$-algebra over $R^\circ$}.
\end{definition}

\begin{notation}
	Throughout the rest of this paper, we fix a simplicial commutative ring $R$, hence the underlying connective $\bbE_\infty$-algebra $R^\circ$ over $\mathbb{Z}$, and simply denote the functor $\Theta_R$ by $\Theta$. Unless otherwise specified, the notations $R$ and $R^\circ$ will indicate the simplicial commutative ring $R$ and the underlying connective $\bbE_\infty$-ring $R^\circ$, respectively. 
\end{notation}

\begin{pg}
	Before stating our first result, we recall a bit of terminology. 
\end{pg}

\begin{pg}\label{monad}
	Let $\calC$ be an $\infty$-category. Let us regard the $\infty$-category $\Fun(\calC, \calC)$ as a monoidal $\infty$-category endowed with the composition monoidal structure of \cite[p.656]{lurie2017ha}. Using the evaluation map $\Fun(\calC, \calC) \times \calC \rightarrow \calC$, we can regard $\calC$ as left-tensored over $\Fun(\calC, \calC)$, in the sense of \cite[4.2.1.19]{lurie2017ha}. According to \cite[4.7.0.1]{lurie2017ha}, we refer to an algebra object of $\Fun(\calC, \calC)$ (in the sense of \cite[4.1.1.6]{lurie2017ha}) as a \emph{monad on $\calC$}. For each monad $T$ on $\calC$, we let $\LMod_T(\calC)$ denote the $\infty$-category of left $T$-module objects of $\calC$; see \cite[4.2.1.13]{lurie2017ha}. 
\end{pg}

\begin{pg}\label{monadic functors}
	Let $g: \calD \rightarrow \calC$ be a functor between $\infty$-categories. If $g$ admits a left adjoint $f$, then the composite functor $T=g \circ f: \calC \rightarrow \calC$ can be regarded as a monad on $\calC$ by virtue of \cite[4.7.3.3]{lurie2017ha}. According to \cite[4.7.3.4]{lurie2017ha}, the functor $g: \calD \rightarrow \calC$ is said to be \emph{monadic} if it admits a right adjoint and the induced functor $\calD \rightarrow \LMod_T(\calC)$ is an equivalence of $\infty$-categories under which $g$ corresponds to the forgetful functor $\LMod_T(\calC) \rightarrow \calC$. The functor $g$ is said to be \emph{comonadic} if the associated functor $g^{\op}: \calC^{\op} \rightarrow \calD^{\op}$ is monadic. 
\end{pg}

\begin{pg}
	According to \cite[25.1.2.4]{lurie2018sag}, the functor $\Theta_\mathbb{Z}: \CAlg^\Delta \rightarrow \CAlg^{\cn}_\mathbb{Z}$ of \ref{the forgetful functor} admits both left and right adjoints, and is both monadic and comonadic. We have the following generalization of \cite[25.1.2.4]{lurie2018sag}:
\end{pg}

\begin{proposition}\label{the forgetful functor is monadic and comonadic}
	The functor $\Theta: \CAlg^\Delta_R \rightarrow \CAlg^{\cn}_{R^\circ}$ of \emph{\ref{the forgetful functor for a simplicial commutative ring}} is conservative and admits both a left adjoint $\Theta^L$ and a right adjoint $\Theta^R$. Moreover, $\Theta$ is monadic and comonadic (see \emph{\ref{monadic functors}}). 
\end{proposition}

\begin{proof}
	We first note that the $\infty$-category $\CAlg^\Delta_R$ is presentable in the sense of \cite[5.5.0.1]{MR2522659} (see \ref{compact generation}) and that the functor $\Theta$ is conservative (because the functor $\Theta_\mathbb{Z}$ of \ref{the forgetful functor} is conservative by virtue of \cite[25.1.2.2]{lurie2018sag}). Since the functor $\Theta_\mathbb{Z}$ admits left and right adjoints (see \cite[25.1.2.4]{lurie2018sag}), it follows from \cite[5.2.5.1]{MR2522659} that $\Theta$ also admits left and right adjoints. In particular, $\Theta$ commutes with small colimits and limits. Applying the Barr--Beck theorem of \cite[4.7.0.3]{lurie2017ha}, we deduce that the functor $\Theta$ is both monadic and comonadic. 
\end{proof}

\begin{remark}
	The functor $\Theta^L: \CAlg^{\cn}_{R^\circ} \rightarrow \CAlg^\Delta_R$ carries a connective $\bbE_\infty$-algebra $A$ over $R^\circ$ to the simplicial commutative $R$-algebra $\Theta^L_\mathbb{Z}(A) \otimes_{\Theta^L_\mathbb{Z}(R^\circ)}R$, which is the tensor product in the $\infty$-category $\CAlg^\Delta$. In the special case where $R$ is an ordinary commutative ring, the pairs of adjoint functors $(\Theta^L, \Theta)$ and $(\Theta, \Theta^R)$ recover those of \cite[25.1.2.4]{lurie2018sag}. 
\end{remark}

\begin{notation}\label{the infinity-category of module spectra}
	Let $\Mod_{R^\circ}$ denote the \emph{$\infty$-category of $R^\circ$-module spectra} in the sense of \cite[7.1.1.2]{lurie2017ha}. We will regard $\Mod_{R^\circ}$ as equipped with the symmetric monoidal structure of \cite[4.5.2.1]{lurie2017ha}; see also \cite[4.8.2.19]{lurie2017ha}. We will often abuse terminology by referring to $R^\circ$-module spectra as \emph{$R^\circ$-modules}. We let $\Mod_{R^\circ}^{\cn}$ denote the full subcategory of $\Mod_{R^\circ}$ spanned by those $R^\circ$-modules $M$ such that $\pi_nM \simeq 0$ for every integer $n <0$; see \cite[7.1.1.10]{lurie2017ha}. We will refer to the objects of $\Mod_{R^\circ}^{\cn}$ as \emph{connective $R^\circ$-modules}.
\end{notation}

\begin{notation}\label{free algebras}
	We let $\Sym^\ast_{R^\circ}$ and $\LSym^\ast_R$ denote left adjoints to the forgetful functor $G: \CAlg^{\cn}_{R^\circ} \rightarrow \Mod^{\cn}_{R^\circ}$ and the composite functor $\CAlg^\Delta_R \stackrel{\Theta}{\rightarrow} \CAlg^{\cn}_{R^\circ} \stackrel{G}{\rightarrow} \Mod^{\cn}_{R^\circ}$, respectively (see \cite[3.1.3.14]{lurie2017ha} and \cite[25.2.2.6]{lurie2018sag}). If $M$ is a connective $R^\circ$-module, we will refer to $\Sym^\ast_{R^\circ}M$ and $\LSym^\ast_R M$ as the \emph{free $\bbE_\infty$-algebra over $R^\circ$ generated by $M$} and the \emph{derived symmetric $R$-algebra on $M$}, respectively.
\end{notation}

\begin{remark}\label{the adjunction for the derived symmetric power functor}
	By virtue of \ref{the forgetful functor is monadic and comonadic}, we can identify $\LSym^\ast_R$ with the composite functor $\Theta^L \circ \Sym^\ast_{R^\circ}$. Using this identification, the counit map $\Sym^\ast_{R^\circ} \circ G \rightarrow \id$ and the unit map $\id \rightarrow \Theta \circ \Theta^L$ determine a canonical map $\alpha_A:\Theta(\LSym^\ast_R(G(A)))\otimes_{\Sym^\ast_{R^\circ} (G(A))}A \rightarrow \Theta(\Theta^L(A))$ in $\CAlg^{\cn}_{R^\circ}$ for every connective $\bbE_\infty$-algebra $A$ over $R^\circ$. 
\end{remark}

\begin{pg}
	According to \ref{the forgetful functor is monadic and comonadic}, we can regard $\CAlg^\Delta_R$ as the $\infty$-category of left modules over the monad $\Theta \circ \Theta^L$ on $\CAlg^{\cn}_{R^\circ}$. We now give a more concrete description of the monad $\Theta \circ \Theta^L$: 
\end{pg}

\begin{theorem}\label{comparison of derived rings}
	Let $A$ be a connective $\bbE_\infty$-algebra over $R^\circ$. Then the canonical map $\alpha_A: \Theta(\LSym^\ast_R(G(A)))\otimes_{\Sym^\ast_{R^\circ} (G(A))}A \rightarrow \Theta(\Theta^L(A))$ of \emph{\ref{the adjunction for the derived symmetric power functor}} is an equivalence of $\bbE_\infty$-algebras over $R^\circ$. 
\end{theorem}

\begin{proof}
	Since the forgetful functor $G: \CAlg^{\cn}_{R^\circ} \rightarrow \Mod^{\cn}_{R^\circ}$ is conservative (see \cite[3.2.2.6]{lurie2017ha}), it will suffice to show that the map $G(\alpha_A)$ is an equivalence of $R^\circ$-modules. We first note that the functors $G$ and $\Theta$ are conservative and preserve geometric realizations of simplicial objects (see \cite[3.2.3.2]{lurie2017ha} and \ref{the forgetful functor is monadic and comonadic}). It follows from the Barr--Beck theorem of \cite[4.7.0.3]{lurie2017ha} that we can identify $\CAlg^{\cn}_{R^\circ}$ and $\CAlg^\Delta_R$ with the $\infty$-categories of left module objects over the monads $(G \circ \Theta)\circ \LSym^\ast_R$ and $G \circ \Sym^\ast_{R^\circ}$ on $\Mod^{\cn}_{R^\circ}$, respectively. Under these identifications, the functor $\theta$ corresponds to the forgetful functor $F: \LMod_{G \circ \Theta \circ \Theta^L \circ \Sym^\ast_{R^\circ}}(\Mod^{\cn}_{R^\circ}) \rightarrow \LMod_{G \circ \Sym^\ast_{R^\circ}}(\Mod^{\cn}_{R^\circ})$ associated to the map of monads $G \circ \Sym^\ast_{R^\circ} \rightarrow G \circ \Theta \circ \Theta^L \circ \Sym^\ast_{R^\circ}$ which is determined by the unit map $\id \rightarrow \Theta \circ \Theta^L$; in particular, we can identify the functor $\Theta^L$ with the extension of scalars along the map of monads $G \circ \Sym^\ast_{R^\circ} \rightarrow G \circ \Theta \circ \Theta^L \circ \Sym^\ast_{R^\circ}$ on $\Mod^{\cn}_{R^\circ}$. Combining this with the fact that the forgetful functor $G$ preserves tensor products (see \cite[3.2.4.4]{lurie2017ha}), we conclude that $G(\alpha_A)$ is an equivalence in $\Mod^{\cn}_{R^\circ}$ as desired.
\end{proof}

\begin{corollary}\label{connected components of the unit map}
	Let $A$ be a connective $\bbE_\infty$-algebra over $R^\circ$. Then the unit map $A \rightarrow \Theta(\Theta^L(A))$ for the adjunction between $\Theta^L$ and $\Theta$ of \emph{\ref{the forgetful functor is monadic and comonadic}} induces an isomorphism $\pi_0A \rightarrow \pi_0(\Theta(\Theta^L(A)))$.
\end{corollary}

\begin{proof}
	Using \ref{comparison of derived rings}, we see that $\pi_0\Theta(\Theta^L(A))$ can be identified with the usual tensor product of $\pi_0\Theta(\LSym^\ast_R(G(A)))$ and $\pi_0A$ over $\pi_0\Sym^\ast_{R^\circ} (G(A))$. The desired result follows from the observation that both $\pi_0 \Theta(\LSym^\ast_R(G(A)))$ and $\pi_0\Sym^\ast_{R^\circ} (G(A))$ are equivalent to the usual symmetric algebra on $\pi_0A$, which we regard as a $\pi_0R$-module. 
\end{proof}

\begin{corollary}\label{counit map and injection on connected components} 
	For every simplicial commutative $R$-algebra $A$, the counit map $\Theta(\Theta^R(A)) \rightarrow A$ for the adjunction between $\Theta$ and $\Theta^R$ of \emph{\ref{the forgetful functor is monadic and comonadic}} is injective on connected components.
\end{corollary}

\begin{proof}
	Applying \ref{comparison of derived rings} to $R^\circ$, we obtain a pullback square of spaces
$$
\Pull{\ast}{\ast}{\Map_{\Mod_{R^\circ}}(R^\circ, G(\Theta(\Theta^R(A)))}{\Map_{\Mod_{R^\circ}}(R^\circ, G(A)),}{}{}{}{}
$$
where the vertical maps correspond to the zero maps. The desired result follows immediately from the long exact sequence of homotopy groups, since we can identify the bottom horizontal map with the map between the underlying spaces of $\Theta(\Theta^R(A))$ and $A$.
\end{proof}

\section{Comparison of Derived and Spectral Deligne-Mumford Stacks}

\begin{pg}
	The main goal of this section is to compare the theories of derived and spectral Deligne-Mumford stacks. We begin by studying sheaves on an $\infty$-topos taking values in the $\infty$-categories $\CAlg^\Delta_R$ and $\CAlg^{\cn}_{R^\circ}$.
\end{pg}

\begin{notation}
	Let $\calX$ be an $\infty$-topos (see \cite[6.1.0.4]{MR2522659}). If $\calC$ is an $\infty$-category, we let $\Shv_{\calC}(\calX)$ denote the \emph{$\infty$-category of $\calC$-valued sheaves on $\calX$}: that is, the full subcategory of $\Fun(\calX^{\op}, \calC)$ spanned by those functors which preserve small limits; see \cite[1.3.1.4]{lurie2018sag}. 
\end{notation}

\begin{pg}\label{comparison for sheaves on derived rings}
	The adjunction $\Adjoint{\Theta^L}{\CAlg^{\cn}_{R^\circ}}{\CAlg^\Delta_R}{\Theta}$ of \ref{the forgetful functor is monadic and comonadic} determines a pair of adjoint functors $\Adjoint{\Psi^L_{\calX}}{\Shv_{\CAlg^{\cn}_{R^\circ}}(\calX)}{\Shv_{\CAlg^\Delta_R}(\calX)}{\Psi_{\calX}}$, where the right adjoint $\Psi_{\calX}$ is given by postcomposition with $\Theta$. To describe the left adjoint, we note that the $\infty$-category $\CAlg^{\cn}_{R^\circ}$ is presentable in the sense of \cite[5.5.0.1]{MR2522659}, so that the inclusion functor $\Shv_{\CAlg^\Delta_R}(\calX) \rightarrow \Fun(\calX^{\op}, \CAlg^\Delta_R)$ admits a left adjoint $L_{\calX}: \Fun(\calX^{\op}, \CAlg^\Delta_R) \rightarrow \Shv_{\CAlg^\Delta_R}(\calX)$, which we refer to as the \emph{sheafification functor}; see \cite[1.3.4.3]{lurie2018sag}. The functor $\Psi^L_{\calX}$ carries each sheaf $\calO: \calX^{\op} \rightarrow \CAlg^{\cn}_{R^\circ}$ to $L_{\calX}(\Theta^L \circ \calO)$. 
\end{pg}

\begin{remark}\label{conservative at the level of sheaves}
	The functor $\Psi_{\calX}$ is conservative, since the functor $\Theta$ is conservative (see \ref{the forgetful functor is monadic and comonadic}).
\end{remark}

\begin{pg}\label{adjunction between sheaves on topoi taking values in a derived ring} 
	Let $\calC$ be a compactly generated $\infty$-category, in the sense of \cite[5.5.7.1]{MR2522659}. Let $f^\ast: \calY \rightarrow \calX$ be a \emph{geometric morphism} of $\infty$-topoi: that is, it preserves small colimits and finite limits. Then precomposition with $f^\ast$ determines a functor $\Shv_{\calC}(\calX) \rightarrow \Shv_{\calC}(\calY)$ which admits a left adjoint by virtue of \cite[21.2.2.4]{lurie2018sag}. Applying this observation to the compactly generated $\infty$-categories $\CAlg^\Delta_R$ and $\CAlg^{\cn}_{R^\circ}$ (see \ref{compact generation} and \cite[7.2.4.27]{lurie2017ha}), we obtain pairs of adjoint functors $\Adjoint{f_\Delta^\ast}{\Shv_{\CAlg^\Delta_R}(\calY)}{\Shv_{\CAlg^\Delta_R}(\calX)}{f^\Delta_\ast}$ and $\Adjoint{f_{\bbE_\infty}^\ast}{\Shv_{\CAlg^{\cn}_{R^\circ}}(\calY)}{\Shv_{\CAlg^{\cn}_{R^\circ}}(\calX)}{f^{\bbE_\infty}_\ast}$. Note that we have an equivalence of functors $\Psi_{\calY} \circ f^\Delta_\ast \simeq f^{\bbE_\infty}_\ast \circ \Psi_{\calX}$. 
\end{pg}

\begin{pg}
	For later use, we record the following:
\end{pg}

\begin{lemma}\label{left adjointable}
	Let $f^\ast: \calY \rightarrow \calX$ be a geometric morphism of $\infty$-topoi. Then the diagram 
$$
\Pull{\Shv_{\CAlg^\Delta_R}(\calX)}{\Shv_{\CAlg^\Delta_R}(\calY)}{\Shv_{\CAlg^{\cn}_{R^\circ}}(\calX)}{\Shv_{\CAlg^{\cn}_{R^\circ}}(\calY)}{f^\Delta_\ast}{\Psi_{\calX}}{\Psi_{\calY}}{f^{\bbE_\infty}_\ast}
$$
is left adjointable: that is, the canonical map $f_{\bbE_\infty}^\ast \circ \Psi_{\calY} \rightarrow \Psi_{\calX} \circ f_\Delta^\ast$ is an equivalence of functors from $\Shv_{\CAlg^\Delta_R}(\calY)$ to $\Shv_{\CAlg^{\cn}_{R^\circ}}(\calX)$. 
\end{lemma}

\begin{proof}
	We wish to show that the diagram
$$
\xymatrix{
\Shv_{\CAlg^\Delta_R}(\calY) \ar[rr]^-{\Psi_{\calY}} \ar[d]^-{f_\Delta^\ast} && \Shv_{\CAlg^{\cn}_{R^\circ}}(\calY) \ar[d]^-{f_{\bbE_\infty}^\ast} \\
\Shv_{\CAlg^\Delta_R}(\calX) \ar[rr]^-{\Psi_{\calX}} && \Shv_{\CAlg^{\cn}_{R^\circ}}(\calX)
}
$$
commutes. Let $\Fun^\omega(\CAlg^\Delta_R, \SSet) \subseteq \Fun(\CAlg^\Delta_R, \SSet)$ and $\Fun^\omega(\CAlg^{\cn}_{R^\circ}, \SSet) \subseteq \Fun(\CAlg^{\cn}_{R^\circ}, \SSet)$ denote the full subcategories spanned by those functors which preserve small filtered colimits. Let $\Fun^\ast(\Fun^\omega(\CAlg^\Delta_R, \SSet), \calY)$ and $\Fun^\ast(\Fun^\omega(\CAlg^{\cn}_{R^\circ}, \SSet), \calY)$ denote the $\infty$-categories of geometric morphisms from $\Fun^\omega(\CAlg^\Delta_R, \SSet)$ and $\Fun^\omega(\CAlg^{\cn}_{R^\circ}, \SSet)$ to $\calY$, respectively (see \cite[21.1.1.2]{lurie2018sag}). Let $T_0: \Fun^\omega(\CAlg^{\cn}_{R^\circ}, \SSet) \rightarrow \Fun^\omega(\CAlg^\Delta_R, \SSet)$ denote the functor given by composition with $\Theta$. According to \cite[21.1.4.3]{lurie2018sag}, $\Fun^\omega(\CAlg^\Delta_R, \SSet)$ and $\Fun^\omega(\CAlg^{\cn}_{R^\circ}, \SSet)$ are classifying $\infty$-topoi for $\CAlg^\Delta_R$-valued and $\CAlg^{\cn}_{R^\circ}$-valued sheaves in the sense of \cite[21.1.0.1]{lurie2018sag}, respectively. Combining this with \cite[21.1.3.3]{lurie2018sag} and \cite[21.1.4.2]{lurie2018sag}, 
we see that the above diagram can be identified with the commutative diagram
$$
\xymatrix{
\Fun^\ast(\Fun^\omega(\CAlg^\Delta_R, \SSet), \calY) \ar[rr]^-{\circ T_0} \ar[d]^-{f^\ast \circ} && \Fun^\ast(\Fun^\omega(\CAlg^{\cn}_{R^\circ}, \SSet), \calY) \ar[d]^-{f^\ast \circ} \\
\Fun^\ast(\Fun^\omega(\CAlg^\Delta_R, \SSet), \calX) \ar[rr]^-{\circ T_0} && \Fun^\ast(\Fun^\omega(\CAlg^{\cn}_{R^\circ}, \SSet), \calX),
}
$$
where the horizontal maps are given by composition with $T_0$ and the vertical maps are given by composition with $f^\ast$, thereby completing the proof.
\end{proof}

\begin{pg}
	In order to extend the adjunction between the pair of adjoint functors $(\Theta^L, \Theta)$ of \ref{the forgetful functor is monadic and comonadic} to an adjunction between the $\infty$-category of derived Deligne-Mumford stacks over $R$ and the $\infty$-category of spectral Deligne-Mumford stacks over $R^\circ$, we establish some preliminaries.
\end{pg}

\begin{notation}\label{sheaves on infinity-topoi}
	Let $\iTop$ denote the $\infty$-category whose objects are $\infty$-topoi $\calX$ and whose morphisms are geometric morphisms $f_\ast: \calX \rightarrow \calY$ (that is, functors which admit a left exact left adjoint); see \cite[6.3.1.5]{MR2522659}. If $\calC$ is an $\infty$-category, we let $\iTop_{\calC}$ denote the $\infty$-category whose objects are pairs $(\calX, \calO_{\calX})$, where $\calX$ is an $\infty$-topos and $\calO_{\calX}$ is a $\calC$-valued sheaf on $\calX$, and whose morphism from $(\calX, \calO_{\calX})$ to $(\calY, \calO_{\calY})$ is a pair $(f_\ast, \alpha)$, where $f_\ast: \calX \rightarrow \calY$ is a geometric morphism of $\infty$-topoi and $\alpha: \calO_{\calY} \rightarrow \calO_{\calX} \circ f^\ast$ is a morphism in $\Shv_{\calC}(\calY)$; see \cite[21.4.1.1]{lurie2018sag}. Let $\sfX=(\calX, \calO)$ be an object of $\iTop_{\calC}$ and let $U \in \calX$ be an object. Let $\calO|_U$ denote the composite of $\calO$ with the forgetful functor $(\calX_{/U})^{\op} \rightarrow \calX^{\op}$, where $\calX_{/U}$ denotes the $\infty$-topos of \cite[6.3.5.1]{MR2522659}. We let $\sfX_U$ denote the object $(\calX_{/U}, \calO|_U) \in \iTop_{\calC}$. 
\end{notation}

\begin{pg}\label{comparison for derived ringed infinity-topoi}
	The pair of adjoint functors $\Adjoint{\Theta^L}{\CAlg^{\cn}_{R^\circ}}{\CAlg^\Delta_R}{\Theta}$ of \ref{the forgetful functor is monadic and comonadic} determines a pair of adjoint functors $\Adjoint{\Xi}{\iTop_{\CAlg^\Delta_R}}{\iTop_{\CAlg^{\cn}_{R^\circ}}}{\Xi^R}$. We note that the left adjoint $\Xi$ carries an object $(\calX, \calO: \calX^{\op} \rightarrow \CAlg^\Delta_R) \in \iTop_{\CAlg^\Delta_R}$ to the pair $(\calX, \Psi_{\calX}(\calO))$ and the right adjoint $\Xi^R$ carries an object $(\calX, \calO: \calX^{\op} \rightarrow \CAlg^{\cn}_{R^\circ}) \in \iTop_{\CAlg^{\cn}_{R^\circ}}$ to the pair $(\calX, \Psi^L_{\calX}(\calO))$; see \ref{comparison for sheaves on derived rings}. 
\end{pg}

\begin{notation}\label{dependence on the base}
	We will denote the functors $\Xi$ and $\Xi^R$ of \ref{comparison for derived ringed infinity-topoi} by $\Xi_R$ and $\Xi^R_R$, respectively, if we wish to emphasize the dependence on $R$. 
\end{notation}

\begin{remark}\label{conservative at the level of ringed infinity-topoi}
	The functor $\Xi: \iTop_{\CAlg^\Delta_R} \rightarrow \iTop_{\CAlg^{\cn}_{R^\circ}}$ is conservative by virtue of \ref{conservative at the level of sheaves}.
\end{remark}

\begin{pg}
	We define \'etale morphisms in $\iTop_{\CAlg^\Delta_R}$ as in the spectral setting of \cite[1.4.10.1]{lurie2018sag}: 
\end{pg}

\begin{definition}
	Let $f: (\calX, \calO_{\calX}) \rightarrow (\calY, \calO_{\calY})$ be a morphism in $\iTop_{\CAlg^\Delta_R}$. We will say that $f$ is \emph{\'etale} if the underlying geometric morphism of $\infty$-topoi $f_\ast: \calX \rightarrow \calY$ is \'etale in the sense of \cite[p.617]{MR2522659} and the induced map $f_\Delta^\ast \calO_{\calY} \rightarrow \calO_{\calX}$ is an equivalence in $\Shv_{\CAlg^\Delta_R}(\calX)$, where $f_\Delta^\ast: \Shv_{\CAlg^\Delta_R}(\calY) \rightarrow \Shv_{\CAlg^\Delta_R}(\calX)$ is the pullback functor of \ref{adjunction between sheaves on topoi taking values in a derived ring}. 
\end{definition}

\begin{remark}\label{preservation of etale morphisms}
	Let $\sfX=(\calX, \calO_{\calX}) \in \iTop_{\CAlg^\Delta_R}$ and $\sfY=(\calY, \calO_{\calY}) \in \iTop_{\CAlg^{\cn}_{R^\circ}}$ be objects. If $U \in \calX$ and $V \in \calY$ are objects, then we have canonical equivalences $(\Xi(\sfX))_U \simeq \Xi(\sfX_U)$ and $(\Xi^R(\sfY))_V \simeq \Xi^R(\sfY_V)$; see \ref{sheaves on infinity-topoi}. It follows from the description of $\Xi$ and $\Xi^R$ supplied by \ref{comparison for derived ringed infinity-topoi} that both the functors $\Xi$ and $\Xi^R$ preserve \'etale morphisms. 
\end{remark}

\begin{pg}
	We recall from \cite[1.4.2.1]{lurie2018sag} the notions of a \emph{local morphism} and a \emph{strictly Henselian object} in the $\infty$-category of sheaves on an $\infty$-topos taking values in $\CAlg^{\cn}_{R^\circ}$, and define the notions in the derived setting.
\end{pg}

\begin{pg}
	Let $\calX$ be an $\infty$-topos and $\calO \in \Shv_{\CAlg^{\cn}_{R^\circ}}(\calX)$ be an object. Let $\Sp$ denote the \emph{$\infty$-category of spectra} in the sense of \cite[1.4.3.1]{lurie2017ha}. Let $\Omega^\infty: \CAlg^{\cn}_{R^\circ} \rightarrow \SSet$ be the composition of the forgetful functor $\CAlg^{\cn}_{R^\circ} \rightarrow \Sp$ with the underlying space functor $\Sp \rightarrow \SSet$ of \cite[1.4.2.20]{lurie2017ha}. Since the functor $\Omega^\infty$ preserves small limits, $\Omega^\infty \circ \calO$ is an object of $\Shv_{\SSet}(\calX) \simeq \calX$. We let $\pi_0\calO \in \tau_{\leq 0}\calX$ denote the $0$-truncation of this object, which we will regard as an object of the ordinary topos $\tau_{\leq 0}\calX$. 
\end{pg}

\begin{pg}\label{strictly Henselian objects and local morphisms in the spectral setting}
	Let $\calX$ be an $\infty$-topos. According to \cite[1.4.2.1]{lurie2018sag}, an object $\calO \in \Shv_{\CAlg^{\cn}_{R^\circ}}(\calX)$ is said to be \emph{strictly Henselian} if $\pi_0\calO$ is strictly Henselian (when regarded as a commutative ring object of the ordinary topos $\tau_{\leq 0}\calX$) in the sense of \cite[1.2.2.5]{lurie2018sag}. A morphism $\calO' \rightarrow \calO$ in $\Shv_{\CAlg^{\cn}_{R^\circ}}(\calX)$ is said to be \emph{local} if the induced morphism $\pi_0\calO' \rightarrow \pi_0\calO$ of commutative ring objects of the ordinary topos $\tau_{\leq 0}\calX$ is local in the sense of \cite[1.2.1.4]{lurie2018sag}. We will denote by $\iTop_{\CAlg^{\cn}_{R^\circ}}^{\sHen}$ the subcategory of $\iTop_{\CAlg^{\cn}_{R^\circ}}$ whose objects are pairs $(\calX, \calO_{\calX})$ such that $\calO_{\calX}$ is strictly Henselian, and whose morphisms are maps $f: (\calX, \calO_{\calX}) \rightarrow (\calY, \calO_{\calY})$ for which the induced map $f_{\bbE_\infty}^\ast \calO_{\calY} \rightarrow \calO_{\calX}$ is local. 
\end{pg}

\begin{definition}\label{strictly Henselian objects and local morphisms in the derived setting}
	Let $\calX$ be an $\infty$-topos. Let $\calO \in \Shv_{\CAlg^\Delta_R}(\calX)$ be an object. We will say that $\calO$ is \emph{strictly Henselian} if $\Psi_{\calX}(\calO) \in \Shv_{\CAlg^{\cn}_{R^\circ}}(\calX)$ is strictly Henselian in the sense of \ref{strictly Henselian objects and local morphisms in the spectral setting} (see \ref{comparison for sheaves on derived rings}). We will say that a morphism $\calO' \rightarrow \calO$ in $\Shv_{\CAlg^\Delta_R}(\calX)$ is \emph{local} if the induced morphism $\Psi_{\calX}(\calO') \rightarrow \Psi_{\calX}(\calO)$ in $\Shv_{\CAlg^{\cn}_{R^\circ}}(\calX)$ is local in the sense of \ref{strictly Henselian objects and local morphisms in the spectral setting}. We will denote by $\iTop_{\CAlg^\Delta_R}^{\sHen}$ the subcategory of $\iTop_{\CAlg^\Delta_R}$ whose objects are pairs $(\calX, \calO_{\calX})$ for which $\calO_{\calX}$ is strictly Henselian, and whose morphisms are maps $f: (\calX, \calO_{\calX}) \rightarrow (\calY, \calO_{\calY})$ such that the induced map $f_\Delta^\ast \calO_{\calY} \rightarrow \calO_{\calX}$ is local.
\end{definition}

\begin{lemma}\label{restriction to the subcategory of strictly Henselian objects and local morphisms}
	The pair of functors $\Xi$ and $\Xi^R$ of \emph{\ref{comparison for derived ringed infinity-topoi}} restricts to a pair of adjoint functors $\adjoint{\iTop_{\CAlg^\Delta_R}^{\sHen}}{\iTop_{\CAlg^{\cn}_{R^\circ}}^{\sHen}}$ between the subcategories spanned by strictly Henselian objects and local morphisms between them.
\end{lemma}

\begin{proof}
	According to \ref{left adjointable}, the functor $\Xi$ restricts to $\iTop_{\CAlg^\Delta_R}^{\sHen} \rightarrow \iTop_{\CAlg^{\cn}_{R^\circ}}^{\sHen}$. To prove that $\Xi^R$ restricts to $\iTop_{\CAlg^{\cn}_{R^\circ}}^{\sHen} \rightarrow \iTop_{\CAlg^\Delta_R}^{\sHen}$, it will suffice to show that if $\calX$ is an $\infty$-topos and $\calO \in \Shv_{\CAlg^{\cn}_{R^\circ}}(\calX)$ is an object, then $\pi_0 \Psi_{\calX}(\Psi^L_{\calX}(\calO))$ can be identified with $\pi_0 \calO$. This follows from \ref{connected components of the unit map}. 
\end{proof}

\begin{pg}\label{etale topology on derived rings}
	We define an \emph{\'etale spectrum} of a simplicial commutative $R$-algebra as in the spectral setting of \cite[1.4.2.5]{lurie2018sag}. Let $A$ be a simplicial commutative $R$-algebra and let $\CAlg^{\Delta, \et}_A \subseteq \CAlg^\Delta_A$ denote the full subcategory spanned by the \'etale $A$-algebras. We will regard $(\CAlg^{\Delta, \et}_A)^{\op}$ as equipped with the \emph{\'etale topology}: that is, if $B$ is an \'etale $A$-algebra, then a sieve $\calC \subseteq (\CAlg^{\Delta, \et}_A)^{\op}_{/B} \simeq (\CAlg^{\Delta, \et}_B)^{\op}$ is a covering if and only if it contains a finite collection of maps $\{B \rightarrow B_i\}_{1 \leq i \leq n}$ for which the induced map $B \rightarrow \prod B_i$ is faithfully flat. Note that the forgetful functor $\CAlg^{\Delta, \et}_A \rightarrow \CAlg^\Delta_R$ determines a $\CAlg^\Delta_R$-valued sheaf on the $\infty$-category $\Shv((\CAlg^{\Delta, \et}_A)^{\op})$ of sheaves on $(\CAlg^{\Delta, \et}_A)^{\op}$; see \cite[1.3.1.7]{lurie2018sag}.
Let us denote this sheaf by $\calO_A$. Then the pair $(\Shv((\CAlg^{\Delta, \et}_A)^{\op}), \calO_A) \in \iTop_{\CAlg^\Delta_R}$ belongs to the subcategory $\iTop_{\CAlg^\Delta_R}^{\sHen}$ by virtue of \cite[4.3.12]{MR2717174}.
\end{pg}

\begin{definition}\label{affine derived Deligne-Mumford stacks}
	Let $A$ be a simplicial commutative $R$-algebra. We will denote the object $(\Shv((\CAlg^{\Delta, \et}_A)^{\op}), \calO_A) \in \iTop_{\CAlg^\Delta_R}^{\sHen}$ of \ref{etale topology on derived rings} by $\Spec A$ and refer to it as the \emph{\'etale spectrum of $A$ over $R$}. 
\end{definition}

\begin{remark}\label{global sections functor is a left adjoint to Spec} 
	The global section functor $\Gamma: \iTop_{\CAlg^\Delta_R}^{\sHen} \rightarrow (\CAlg^\Delta_R)^{\op}$ which carries an object $(\calX, \calO)$ to $\calO(\ast)$, where $\ast \in \calX$ is a final object, admits a right adjoint. Concretely, it is given by the construction $A \mapsto \Spec A$; see \cite[4.1.16]{MR2717174} and \cite[4.3.12]{MR2717174}.
\end{remark}

\begin{pg}\label{the equivalence of etale sites}
	If $A$ is a connective $\bbE_\infty$-ring, we let $\CAlg^{\et}_A$ denote the full subcategory of $\CAlg^{\cn}_A$ spanned by the $\bbE_\infty$-rings which are \'etale over $A$. According to \cite[3.4.13]{MR2717174} and \cite[7.5.4.2]{lurie2017ha}, the truncation functors $\CAlg^{\Delta, \et}_R \rightarrow  \CAlg^{\et}_{\pi_0R} \leftarrow \CAlg^{\et}_{R^\circ}$ are equivalences of $\infty$-categories, so that the functor $\Theta: \CAlg^\Delta_R \rightarrow \CAlg^{\cn}_{R^\circ}$ of \ref{the forgetful functor for a simplicial commutative ring} restricts to an equivalence of $\infty$-categories $\CAlg^{\Delta, \et}_R \rightarrow \CAlg^{\et}_{R^\circ}$. In particular, we have an equivalence of $\infty$-topoi $\Shv((\CAlg^{\Delta, \et}_A)^{\op}) \simeq \Shv((\CAlg^{\et}_{A^\circ})^{\op})$, where we regard $(\CAlg^{\et}_{A^\circ})^{\op}$ as equipped with the \'etale topology of \cite[B.6.2.2]{lurie2018sag}.
\end{pg}

\begin{pg}
	As in the case of (affine) spectral Deligne-Mumford stacks (see \cite[1.4.4.2]{lurie2018sag} and \cite[1.4.7]{lurie2018sag}), we define (\emph{affine}) \emph{derived Deligne-Mumford stacks} as follows: 
\end{pg}

\begin{definition}
	Let $\sfX$ be an object of $\iTop_{\CAlg^\Delta_R}^{\sHen}$. We will say that $\sfX$ is an \emph{affine derived Deligne-Mumford stack over $R$} if there exists an equivalence $\sfX \simeq \Spec A$ for some simplicial commutative $R$-algebra $A$; see \ref{affine derived Deligne-Mumford stacks}. We will say that $\sfX$ is a \emph{derived Deligne-Mumford stack over $R$} if there exists a collection of objects $\{U_\alpha \}$ of $\calX$ for which the map $\coprod U_\alpha \rightarrow \ast$ is an effective epimorphism in $\calX$, where $\ast \in \calX$ is a final object, and each $\sfX_{U_\alpha}$ is an affine derived Deligne-Mumford stack over $R$; see \ref{sheaves on infinity-topoi}.
\end{definition}

\begin{notation}\label{the infinity-categories of derived and spectral Deligne-Mumford stacks}
	We let $\DerDM_R$ denote the full subcategory of $\iTop_{\CAlg^\Delta_R}^{\sHen}$ spanned by the derived Deligne-Mumford stacks over $R$, and we let $\SpDM_{R^\circ}$ denote the \emph{$\infty$-category of spectral Deligne-Mumford stacks over $R^\circ$} (see \cite[1.4.4.2]{lurie2018sag}). In the special case where $R=\mathbb{Z}$, we will denote $\DerDM_R$ simply by $\DerDM$ and refer to it as the \emph{$\infty$-category of derived Deligne-Mumford stacks}.
\end{notation}

\begin{remark}\label{undercategories}
	We have canonical equivalences of $\infty$-categories $\DerDM_R \simeq \DerDM_{/\Spec R}$ and $\SpDM_{R^\circ} \simeq (\SpDM_\mathbb{Z})_{/\Spec R^\circ}$. 
\end{remark}

\begin{pg}
	The following definition is a derived analogue of the notion of \emph{spectral Deligne-Mumford $n$-stacks} of \cite[1.6.8.1]{lurie2018sag}: 
\end{pg}

\begin{definition}\label{n-stacks}
	Let $n \geq 0$ be an integer. Let $\sfX$ be a derived Deligne-Mumford stack. We will say that $\sfX$ is a \emph{derived Deligne-Mumford $n$-stack} if the space $\Map_{\DerDM}(\Spec A, \sfX)$ is $n$-truncated for each ordinary commutative ring $A$. We will refer to derived Deligne-Mumford $0$-stacks as \emph{derived algebraic spaces}. 
\end{definition}

\begin{pg}
	We now show that the adjunction $\Adjoint{\Xi}{\iTop_{\CAlg^\Delta_R}^{\sHen}}{\iTop_{\CAlg^{\cn}_{R^\circ}}^{\sHen}}{\Xi^R}$ of \ref{restriction to the subcategory of strictly Henselian objects and local morphisms} restricts to the full subcategories $\DerDM_R$ and $\SpDM_{R^\circ}$. 
\end{pg}

\begin{lemma}\label{preservation of affine objects}
	Let $A$ be a simplicial commutative $R$-algebra and $B$ be a connective $\bbE_\infty$-ring over $R^\circ$. Then the images of $\Spec A \in \DerDM_R$ and $\Spec B \in \SpDM_{R^\circ}$ under the functors $\Xi: \iTop_{\CAlg^\Delta_R}^{\sHen} \rightarrow \iTop_{\CAlg^{\cn}_{R^\circ}}^{\sHen}$ and $\Xi^R: \iTop_{\CAlg^{\cn}_{R^\circ}}^{\sHen} \rightarrow \iTop_{\CAlg^\Delta_R}^{\sHen}$ of \emph{\ref{restriction to the subcategory of strictly Henselian objects and local morphisms}} are equivalent to the affine spectral and derived Deligne-Mumford stacks $\Spec \Theta(A)$ and $\Spec \Theta^L(B)$ over $R^\circ$ and $R$, respectively.
\end{lemma}

\begin{proof}
	The first assertion follows immediately from the description of $\Xi$ (see \ref{comparison for derived ringed infinity-topoi}) and the equivalence of $\infty$-categories $\CAlg^{\Delta, \et}_R \rightarrow \CAlg^{\et}_{R^\circ}$ of \ref{the equivalence of etale sites}. To prove the second, let $\sfY' \in \iTop_{\CAlg^{\cn}_{R^\circ}}^{\sHen}$ be an object. Let $\Gamma: \iTop_{\CAlg^{\cn}_{R^\circ}}^{\sHen} \rightarrow (\CAlg^{\cn}_{R^\circ})^{\op}$ denote the global sections functor of \cite[1.4.2.3]{lurie2018sag}, so that we have an adjunction $\Adjoint{\Spec}{(\CAlg^{\cn}_{R^\circ})^{\op}}{\iTop_{\CAlg^{\cn}_{R^\circ}}^{\sHen}}{\Gamma}$. To avoid confusion, let us temporarily denote by $\Gamma^\Delta$ the global sections functor $\Gamma: \iTop_{\CAlg^\Delta_R}^{\sHen} \rightarrow (\CAlg^\Delta_R)^{\op}$ of \ref{global sections functor is a left adjoint to Spec}. It follows from the description of $\Xi$ supplied by \ref{comparison for derived ringed infinity-topoi} that there is a canonical equivalence of functors $\Gamma \circ \Xi \simeq \Theta \circ \Gamma^\Delta$ from $\iTop_{\CAlg^\Delta_R}^{\sHen}$ to $(\CAlg^{\cn}_{R^\circ})^{\op}$. Using this observation, we can identify the mapping space $\Map_{\iTop_{\CAlg^{\cn}_{R^\circ}}^{\sHen}}(\Xi(\sfY'), \Spec B)$ with $\Map_{\iTop_{\CAlg^\Delta_R}^{\sHen}}(\sfY', \Spec \Theta^L(B))$, so that $\Xi^R(\Spec B) \simeq \Spec \Theta^L(B)$ as desired.
\end{proof}

\begin{proposition}\label{associated spectral and derived Deligne-Mumford stacks}
	Let $\sfX$ and $\sfY$ be derived and spectral Deligne-Mumford stacks over $R$ and $R^\circ$, respectively. Then the images of $\sfX$ and $\sfY$ under the functors $\Xi: \iTop_{\CAlg^\Delta_R}^{\sHen} \rightarrow \iTop_{\CAlg^{\cn}_{R^\circ}}^{\sHen}$ and $\Xi^R: \iTop_{\CAlg^{\cn}_{R^\circ}}^{\sHen} \rightarrow \iTop_{\CAlg^\Delta_R}^{\sHen}$ of \emph{\ref{restriction to the subcategory of strictly Henselian objects and local morphisms}} are spectral and derived Deligne-Mumford stacks over $R^\circ$ and $R$, respectively. 
\end{proposition}

\begin{proof}
	Let $\sfY=(\calY, \calO_{\calY})$. We will prove that $\Xi^R(\sfY)$ is a derived Deligne-Mumford stack over $R$; the proof that $\Xi(\sfX) \in \SpDM_{R^\circ}$ is similar. Let $\ast \in \calY$ denote a final object, and choose an effective epimorphism $\coprod U_\alpha \rightarrow \ast$ in $\calY$ for which $\sfY_{U_\alpha}$ is an affine spectral Deligne-Mumford stack over $R^\circ$ for every $\alpha$ (see \ref{sheaves on infinity-topoi}). Since we have an equivalence $\Xi^R(\sfY_{U_\alpha}) \simeq (\Xi^R(\sfY))_{U_\alpha}$ for each $\alpha$ by virtue of \ref{preservation of etale morphisms}, the desired result follows immediately from \ref{preservation of affine objects}. 
\end{proof}

\begin{remark}
	The proof of \ref{associated spectral and derived Deligne-Mumford stacks} guarantees that if $\sfY$ is a spectral Deligne-Mumford stack over $R^\circ$ and $\{ \Spec B_\alpha \rightarrow \sfY\}$ is a collection of \'etale morphisms for which the induced map $\coprod \Spec B_\alpha \rightarrow \sfY$ is surjective (see \cite[3.5.5.5]{lurie2018sag}), then $\Xi^R(\sfY)$ is a colimit of the diagram $\{ \Spec \Theta^L(B_\alpha) \}$ in $\DerDM_R$. 
\end{remark}

\begin{remark}\label{the adjunction for derived and spectral Deligne-Mumford stacks}
	According to \ref{associated spectral and derived Deligne-Mumford stacks}, the adjunction $\Adjoint{\Xi}{\iTop_{\CAlg^\Delta_R}^{\sHen}}{\iTop_{\CAlg^{\cn}_{R^\circ}}^{\sHen}}{\Xi^R}$ of \ref{restriction to the subcategory of strictly Henselian objects and local morphisms} restricts to adjoint functors $\adjoint{\DerDM_R}{\SpDM_{R^\circ}}$.
\end{remark}

\begin{definition}\label{underlying spectral Deligne-Mumford stacks}
	Let $\sfX$ be a derived Deligne-Mumford stack over $R$. We will refer to the spectral Deligne-Mumford stack $\Xi(\sfX)$ as the \emph{underlying spectral Deligne-Mumford stack of $\sfX$ over $R^\circ$}.
\end{definition}

\begin{remark}\label{description using the ring of integers}
	Under the equivalences of \ref{undercategories}, we can identify the functor $\Xi: \DerDM_R \rightarrow \SpDM_{R^\circ}$ with the functor induced by $\Xi_\mathbb{Z}: \DerDM \rightarrow \SpDM_\mathbb{Z}$ (see \ref{dependence on the base}). Moreover, if $\sfX$ is a spectral Deligne-Mumford stack over $R^\circ$, then the canonical map $\Xi^R(\sfX) \rightarrow \Xi^R_\mathbb{Z}(\sfX) \times_{\Spec \Theta^L_\mathbb{Z}(\Theta_\mathbb{Z}(R))} \Spec R$ is an equivalence in $\DerDM_R$. 
\end{remark}

\begin{pg}
	Let $n \geq 0$ be an integer. According to \cite[7.1.3.14]{lurie2017ha}, a connective $\bbE_\infty$-algebra $A$ over $R^\circ$ is $n$-truncated as an object of $\CAlg^{\cn}_{R^\circ}$ (in the sense of \cite[5.5.6.1]{MR2522659}) if and only if $\pi_mA \simeq 0$ for every $m>n$ (see \cite[p.1200]{lurie2017ha}). Arguing as in the proof of \cite[7.1.3.14]{lurie2017ha}, we see that a simplicial commutative $R$-algebra $A$ is $n$-truncated as an object of $\CAlg^\Delta_R$ if and only if $\pi_m A \simeq 0$ for every $m>n$; here $\pi_mA$ denotes the $m$th homotopy group of the underlying $\bbE_\infty$-algebra $A^\circ$ over $R^\circ$. We let $\tau_{\leq n}\CAlg^\Delta_R \subseteq \CAlg^\Delta_R$ and $\tau_{\leq n}\CAlg^{\cn}_{R^\circ} \subseteq \CAlg^{\cn}_{R^\circ}$ denote the full subcategories spanned by the $n$-truncated objects. 
\end{pg}

\begin{pg}\label{equivalence of 0-truncations of derived rings}
	Let $n \geq 0$ be an integer. Then the functor $\Theta: \CAlg^\Delta_R \rightarrow \CAlg^{\cn}_{R^\circ}$ restricts to a functor $\Theta_{\tau_{\leq n}}: \tau_{\leq n}\CAlg^\Delta_R \rightarrow \tau_{\leq n}\CAlg^{\cn}_{R^\circ}$. In the special case where $n=0$, the functor $\Theta_{\tau_{\leq n}}$ is an equivalence of $\infty$-categories, because we can identify both the $\infty$-categories $\tau_{\leq 0}\CAlg^\Delta_R$ and $\tau_{\leq 0}\CAlg^{\cn}_{R^\circ}$ with the (nerve of the) ordinary category of commutative algebras over $\pi_0R$. 
\end{pg}

\begin{remark}\label{equivalence of 0-truncations}
	Using \ref{equivalence of 0-truncations of derived rings}, we can identify the \emph{underlying Deligne-Mumford stack} $(\tau_{\leq 0} \calX, \pi_0 \calO)$ of a derived Deligne-Mumford stack $\sfX=(\calX, \calO)$ over $R$ with the \emph{underlying Deligne-Mumford stack} of the spectral Deligne-Mumford stack $\Xi(\sfX)$ over $R^\circ$; see \cite[1.4.8.2]{lurie2018sag}. 
\end{remark}

\begin{definition}\label{truncated stacks}
	Let $n \geq 0$ be an integer. According to \cite[1.4.6.1]{lurie2018sag}, a spectral Deligne-Mumford stack $\sfX=(\calX, \calO_{\calX})$ over $R^\circ$ is said to be \emph{$n$-truncated} if $\calO_{\calX}$ is $n$-truncated in the sense of \cite[1.3.5.5]{lurie2018sag}. We will say that a derived Deligne-Mumford stack over $R$ is \emph{$n$-truncated} if the underlying spectral Deligne-Mumford stack over $R^\circ$ in the sense of \ref{underlying spectral Deligne-Mumford stacks} is $n$-truncated. 
\end{definition}

\begin{remark}\label{equivalence between 0-truncated DerDM and SpDM}
	It follows from \ref{equivalence of 0-truncations of derived rings} that there is an equivalence of $\infty$-categories between the full subcategory of $\DerDM_R$ spanned by those derived Deligne-Mumford stacks which are $0$-truncated and the full subcategory of $\SpDM_{R^\circ}$ spanned by those spectral Deligne-Mumford stacks which are $0$-truncated.
\end{remark}

\begin{pg}
	We now establish some basic properties preserved by the functor $\Xi$. First, we note that the $\infty$-categories $\DerDM_R$ and $\SpDM_{R^\circ}$ admit finite limits (see \cite[1.4.11.1]{lurie2018sag}). 
\end{pg}

\begin{lemma}\label{preservation of fiber products} 
	The functor $\Xi: \DerDM_R \rightarrow \SpDM_{R^\circ}$ preserves finite limits. 
\end{lemma}

\begin{proof}
	Since $\Xi$ carries $\Spec R$ to $\Spec R^\circ$ by virtue of \ref{preservation of affine objects}, it will suffice to show that the functor $\Xi$ preserves fiber products. Suppose we are given a diagram $\sfY \stackrel{f}{\rightarrow} \sfX \stackrel{g}{\leftarrow} \sfZ$ of derived Deligne-Mumford stacks over $R$. Set $\sfX=(\calX, \calO_{\calX})$. Let us say that an object $U \in \calX$ is \emph{good} if the canonical map $\delta_U:  \Xi(\sfY_{f^\ast U} \times_{\sfX_U} \sfZ_{g^\ast U})\rightarrow \Xi(\sfY_{f^\ast U}) \times_{\Xi(\sfX_U)} \Xi(\sfZ_{g^\ast U})$ is an equivalence of spectral Deligne-Mumford stacks over $R^\circ$; see \ref{sheaves on infinity-topoi}. We first show that the collection of good objects of $\calX$ is closed under small colimits. To this end, suppose we are given a diagram of good objects $\{U_\alpha\}$ having a colimit $U \in \calX$; we wish to show that $U$ is good. We claim that $\sfY_{f^\ast U} \times_{\sfX_U} \sfZ_{g^\ast U}$ is a colimit of the diagram $\{ \sfY_{f^\ast U_\alpha} \times_{\sfX_{U_\alpha}} \sfZ_{g^\ast U_\alpha} \}$ in $\DerDM_R$. Since the transition maps $\sfY_{f^\ast U_\alpha} \times_{\sfX_{U_\alpha}} \sfZ_{g^\ast U_\alpha} \rightarrow \sfY_{f^\ast U_\beta} \times_{\sfX_{U_\beta}} \sfZ_{g^\ast U_\beta}$ are \'etale, the diagram has a colimit in the $\infty$-category $\DerDM_R$ (see \cite[21.4.6.4]{lurie2018sag}). Using this observation, the desired result follows from the fact that colimits are universal in the $\infty$-category of \'etale sheaves on $(\CAlg^\Delta_R)^{\op}$; see \cite[2.17]{MR4598184} and \cite[6.1.0.6]{MR2522659}. Similarly, using \ref{preservation of etale morphisms}, we see that $\Xi(\sfY_{f^\ast U}) \times_{\Xi(\sfX_U)} \Xi(\sfZ_{g^\ast U})$ is a colimit of the diagram $\{ \Xi(\sfY_{f^\ast U_\alpha}) \times_{\Xi(\sfX_{U_\alpha})} \Xi(\sfZ_{g^\ast U_\alpha}) \}$ in $\SpDM_{R^\circ}$. Combining these observations with the fact that the functor $\Xi$ preserves a colimit of a diagram with \'etale transition maps in the $\infty$-category $\DerDM_R$ (see \ref{the adjunction for derived and spectral Deligne-Mumford stacks}), we have equivalences
\begin{eqnarray*}
\Xi(\sfY_{f^\ast U} \times_{\sfX_U} \sfZ_{g^\ast U}) 
& \simeq & \Xi(\colim (\sfY_{f^\ast U_\alpha} \times_{\sfX_{U_\alpha}} \sfZ_{g^\ast U_\alpha})) \\
& \simeq & \colim \Xi(\sfY_{f^\ast U_\alpha} \times_{\sfX_{U_\alpha}} \sfZ_{g^\ast U_\alpha}) \\
& \simeq & \colim \Xi(\sfY_{f^\ast U_\alpha}) \times_{\Xi(\sfX_{U_\alpha})} \Xi(\sfZ_{g^\ast U_\alpha}) \\
& \simeq &  \Xi(\sfY_{f^\ast U}) \times_{\Xi(\sfX_U)} \Xi(\sfZ_{g^\ast U});
\end{eqnarray*}
here the third equivalence follows from the assumption that each $U_\alpha$ is good. This completes the proof that $U$ is good. According to the derived analogue of \cite[1.4.7.9]{lurie2018sag}, it will suffice to show that every affine object $U$ of $\calX$ (that is, an object $U \in \calX$ for which $\sfX_U$ is an affine derived Deligne-Mumford stack over $R$) is good. Replacing $\sfX$ by $\sfX_U$, we may assume that $\sfX \simeq \Spec A$ is affine. Arguing as above, we can reduce to the case where $\sfY \simeq \Spec B$ and $\sfZ \simeq \Spec C$ are also affine. Using \ref{preservation of affine objects}, it suffices to show that the canonical map $\Theta(B)\otimes_{\Theta(A)}\Theta(C) \rightarrow \Theta(B \otimes_A C)$ is an equivalence in $\CAlg^{\cn}_{R^\circ}$, which follows from the fact that $\Theta$ preserves small colimits (see \ref{the forgetful functor is monadic and comonadic}).
\end{proof}

\begin{pg}
	In what follows, we consider properties of morphisms of derived and spectral Deligne-Mumford stacks. We refer to \cite{lurie2018sag} for those properties in the spectral setting, which can be defined similarly in the derived setting. 
\end{pg}

\begin{lemma}\label{preservation of properties}
	Let $f: \sfX \rightarrow \sfY$ be a morphism between derived Deligne-Mumford stacks over $R$. Let $n \geq 0$ be an integer. Then: 
\begin{itemize}
\item[$(1)$] The morphism $f$ is \'etale if and only if $\Xi(f)$ is \'etale (see \emph{\cite[1.4.10.1]{lurie2018sag}}).
\item[$(2)$] The morphism $f$ is quasi-compact if and only if $\Xi(f)$ is quasi-compact (see \emph{\cite[2.3.2.2]{lurie2018sag}}).
\item[$(3)$] The morphism $f$ is a closed immersion if and only if $\Xi(f)$ is a closed immersion (see \emph{\cite[3.1.0.1]{lurie2018sag}}).
\item[$(4)$] The morphism $f$ is quasi-separated if and only if $\Xi(f)$ is quasi-separated (see \emph{\cite[3.4.0.1]{lurie2018sag}}). 
\item[$(5)$] The morphism $f$ is separated if and only if $\Xi(f)$ is separated (see \emph{\cite[3.2.0.1]{lurie2018sag}}). 
\item[$(6)$] The morphism $f$ is locally of finite type if and only if $\Xi(f)$ is locally of finite type (see \emph{\cite[4.2.0.1]{lurie2018sag}}). 
\item[$(7)$] The morphism $f$ is a relative derived Deligne-Mumford $n$-stack if and only if $\Xi(f)$ is a relative spectral Deligne-Mumford $n$-stack (see \emph{\cite[6.3.1.11]{lurie2018sag}}).
\item[$(8)$] The morphism $f$ is proper if and only if $\Xi(f)$ is proper (see \emph{\cite[5.1.2.1]{lurie2018sag}}).
\item[$(9)$] The morphism $f$ is locally almost of finite presentation if and only if $\Xi(f)$ is locally almost of finite presentation (see \emph{\cite[4.2.0.1]{lurie2018sag}}).
\end{itemize}
\end{lemma}

\begin{proof}
	We note that $f$ and $\Xi(f)$ have the same underlying morphisms of $\infty$-topoi. Using this together with \ref{left adjointable}, we obtain the assertions (1) through (3). By virtue of \ref{preservation of fiber products}, assertions (4) and (5) follow from (2) and (3), respectively. Using \ref{equivalence of 0-truncations}, assertion (6) follows immediately from \cite[4.2.0.2]{lurie2018sag} and its derived analogue. To prove assertions (7) and (8), note that both the property of being a relative $n$-stack morphism and the property of being a proper morphism are local for the \'etale topology; see \cite[6.3.3.6]{lurie2018sag}. Combining this observation with (1), we may assume that $\sfY$ is affine. We now prove (7). Using \ref{description using the ring of integers}, we can reduce to the case where $R=\mathbb{Z}$. By virtue of \cite[6.3.1.14]{lurie2018sag} and its derived analogue, it will suffice to show that $\sfX$ is a derived Deligne-Mumford $n$-stack if and only if $\Xi_\mathbb{Z}(\sfX)$ is a spectral Deligne-Mumford $n$-stack (see \ref{n-stacks}). Using \cite[1.6.8.3]{lurie2018sag} and its derived analogue, we may suppose that $\sfX$ and $\Xi_\mathbb{Z}(\sfX)$ are $0$-truncated (see \ref{truncated stacks}). In this case, the desired result follows from \ref{equivalence between 0-truncated DerDM and SpDM}. To prove (8), we may also assume that $f$ and $\Xi(f)$ are separated by virtue of (5). Then $\sfX$ and $\Xi(\sfX)$ are derived and spectral algebraic spaces over $R$ and $R^\circ$, respectively (see \cite[3.2.1.1]{lurie2018sag}). The desired result now follows from \cite[5.1.2.2]{lurie2018sag} and its derived analogue. We now prove (9). It follows from \cite[4.2.1.1]{lurie2018sag} and its derived analogue that the property of being a locally almost of finite presentation morphism is local on the source and target with respect to the \'etale topology in both the spectral and derived settings. We may therefore assume that both $\sfX$ and $\sfY$ are affine, in which case the desired result follows immediately from \cite[p.1683]{lurie2018sag}.
\end{proof}

\begin{pg}\label{QCoh in the spectral setting}
	Let $X: \CAlg^{\cn}_{R^\circ} \rightarrow \SSet$ be a functor. We can associate to $X$ an $\infty$-category $\QCoh(X)$ which we refer to as the \emph{$\infty$-category of quasi-coherent sheaves on $X$}; see \cite[6.2.2.1]{lurie2018sag}. An object $F \in \QCoh(X)$ can be described informally as follows: it is a rule which assigns to each pair $(A \in \CAlg^{\cn}_{R^\circ}, \eta \in X(A))$ an $A$-module $F_\eta$, depending functorially on $A$ in the following sense: if $A \rightarrow A'$ is a map in $\CAlg^{\cn}_{R^\circ}$ carrying $\eta$ to $\eta' \in X(A')$, then we have a canonical equivalence of $A'$-modules $F_{\eta'} \simeq F_\eta \otimes_A A'$ (see \cite[6.2.2.7]{lurie2018sag}). We will regard $\QCoh(X)$ as equipped with the symmetric monoidal structure of \cite[6.2.6]{lurie2018sag}, which can be described more informally as follows: $F \otimes G$ assigns the tensor product of $A$-modules $F_\eta \otimes_A G_\eta$ to each $\eta \in X(A)$. 
\end{pg}

\begin{notation}
	Let $\sfX$ be a spectral Deligne-Mumford stack over $R^\circ$. We let $\QCoh(\sfX)$ denote the $\infty$-category of quasi-coherent sheaves on the functor $X: \CAlg^{\cn}_{R^\circ} \rightarrow \SSet$ represented by $\sfX$ (given by the formula $X(A)=\Map_{\SpDM_R}(\Spec A, \sfX)$). We will refer to $\QCoh(\sfX)$ as the \emph{$\infty$-category of quasi-coherent sheaves on $\sfX$}. 
\end{notation}

\begin{definition}\label{quasi-coherent sheaves in the derived setting}
	Let $X: \CAlg^\Delta_R \rightarrow \SSet$ be a functor. We let $\QCoh(X)$ denote the $\infty$-category of quasi-coherent sheaves on the left Kan extension of $X$ along $\Theta: \CAlg^\Delta_R \rightarrow \CAlg^{\cn}_{R^\circ}$ in the sense of \cite[4.3.3.2]{MR2522659}; see \ref{QCoh in the spectral setting}. We will refer to $\QCoh(X)$ as the \emph{$\infty$-category of quasi-coherent sheaves on $X$}. Let $\sfX$ be a derived Deligne-Mumford stack over $R$ and let $X: \CAlg^\Delta_R \rightarrow \SSet$ be the functor represented by $\sfX$, given by the formula $X(A)=\Map_{\DerDM_R}(\Spec A, \sfX)$. We let $\QCoh(\sfX)$ denote the $\infty$-category of quasi-coherent sheaves on $X$ and refer to it as the \emph{$\infty$-category of quasi-coherent sheaves on $\sfX$}. 
\end{definition}

\begin{remark}\label{comparison for QCoh}
	Let $\sfX$ be a derived Deligne-Mumford stack over $R$. Arguing as in \cite[5.4]{MR4598184}, we see that the canonical map $\QCoh(\Xi(\sfX)) \rightarrow \QCoh(\sfX)$ is an equivalence of $\infty$-categories. In particular, if $\sfX=\Spec A$ is an affine derived Deligne-Mumford stack over $R$, then $\QCoh(\sfX)$ can be identified with the $\infty$-category $\Mod_{A^\circ}$ of $A^\circ$-module spectra (see \ref{the infinity-category of module spectra}).
\end{remark}

\begin{notation}\label{almost perfect objects in the derived setting}
	Let $\sfX$ be a derived Deligne-Mumford stack over $R$. We let $\APerf(\sfX)$ denote the full subcategory $\APerf(\Xi(\sfX)) \subseteq \QCoh(\Xi(\sfX))$ spanned by the almost perfect objects under the equivalence of $\infty$-categories $\QCoh(\Xi(\sfX)) \rightarrow \QCoh(\sfX)$ of \ref{comparison for QCoh}; see \cite[6.2.5.3]{lurie2018sag}.
\end{notation}

\section{Comparison of Formal Derived and Spectral Deligne-Mumford Stacks}

\begin{pg}
	Our goal in this section is to develop the theory of formal derived Deligne-Mumford stacks and describe its relationship to the theory of formal spectral Deligne-Mumford stacks. First, let us introduce some terminology. 
\end{pg}

\begin{definition}\label{adic simplicial commutative rings} 
	Let $A$ be a simplicial commutative ring. We will say that $A$ is an \emph{adic simplicial commutative ring} if the underlying commutative ring $\pi_0A$ is an adic ring, in the sense that $\pi_0A$ is equipped with the \emph{$I$-adic topology} for some finitely generated ideal $I \subseteq \pi_0A$ (that is, the topology whose basis of open sets are given by the subsets of the form $a+I^n$, where $a \in A$ and $n \geq 0$ is an integer; see \cite[8.1.1.1]{lurie2018sag}). We will say that $I \subseteq \pi_0A$ is an \emph{ideal of definition}.
\end{definition}

\begin{remark}\label{adic derived rings}
	Let $A$ be an adic simplicial commutative $R$-algebra. Since $\pi_0 \Theta(A) \simeq \pi_0A$, we can regard $\Theta(A)$ as an adic $\bbE_\infty$-algebra over $R^\circ$ with the same ideal of definition as $A$, in the sense of \cite[8.1.1.5]{lurie2018sag}.
\end{remark}

\begin{notation}
	Let $I \subseteq \pi_0R$ be a finitely generated ideal. We let $\Mod_{R^\circ}^{\Cpl(I)} \subseteq \Mod_{R^\circ}$ denote the full subcategory spanned by the \emph{$I$-complete $R^\circ$-modules} in the sense of \cite[7.3.1.1]{lurie2018sag}. We will regard $\Mod_{R^\circ}^{\Cpl(I)}$ as equipped with the symmetric monoidal structure of \cite[7.3.5.6]{lurie2018sag}. We let $(\Mod_{R^\circ}^{\Cpl(I)})_{\geq 0}$ denote the $\infty$-category $\Mod_{R^\circ}^{\Cpl(I)} \cap \Mod_{R^\circ}^{\cn}$; see \ref{the infinity-category of module spectra}. 
\end{notation}

\begin{pg}\label{completion is right t-exact}
	By virtue of \cite[7.3.1.5]{lurie2018sag}, the inclusion functor $\Mod_{R^\circ}^{\Cpl(I)} \rightarrow \Mod_{R^\circ}$ admits a left adjoint, which we refer to as the \emph{$I$-completion functor}. Note that this left adjoint functor carries $\Mod_{R^\circ}^{\cn}$ into $(\Mod_{R^\circ}^{\Cpl(I)})_{\geq 0}$ and is symmetric monoidal (see \cite[7.3.4.4]{lurie2018sag} and \cite[7.3.5.6]{lurie2018sag}). Consequently, $(\Mod_{R^\circ}^{\Cpl(I)})_{\geq 0}$ inherits a symmetric monoidal structure from $\Mod_{R^\circ}^{\Cpl(I)}$ and the functor $\Mod_{R^\circ}^{\cn} \rightarrow (\Mod_{R^\circ}^{\Cpl(I)})_{\geq 0}$ which is a left adjoint to the inclusion functor $(\Mod_{R^\circ}^{\Cpl(I)})_{\geq 0} \rightarrow \Mod_{R^\circ}^{\cn}$ is symmetric monoidal. 
\end{pg}

\begin{definition}\label{complete rings in the derived setting}
	Let $I$ be a finitely generated ideal of $\pi_0R$. We let $\CAlg^{\Delta, \Cpl(I)}_R \subseteq \CAlg^\Delta_R$ and $(\CAlg^{\Cpl(I)}_{R^\circ})_{\geq 0} \subseteq \CAlg^{\cn}_{R^\circ}$ denote the full subcategories spanned by those simplicial commutative $R$-algebras $A$ and connective $\bbE_\infty$-algebras $B$ over $R^\circ$ whose underlying $R^\circ$-modules $G(\Theta(A))$ and $G(B)$ are $I$-complete in the sense of \cite[7.3.1.1]{lurie2018sag}, respectively, where $G: \CAlg^{\cn}_{R^\circ} \rightarrow \Mod^{\cn}_{R^\circ}$ denotes the forgetful functor. We will refer to the objects of $\CAlg^{\Delta, \Cpl(I)}_R$ as \emph{$I$-complete simplicial commutative $R$-algebras}. 

\end{definition}

\begin{notation}\label{complete rings in the spectral setting}
	We let $(\CAlg^{\Cpl(I)}_{R^\circ})_{\geq 0}$ denote the $\infty$-category $\CAlg((\Mod_{R^\circ}^{\Cpl(I)})_{\geq 0})$ and refer to the objects of $(\CAlg^{\Cpl(I)}_{R^\circ})_{\geq 0}$ as \emph{$I$-complete connective $\bbE_\infty$-algebras over $R^\circ$}. Since the functor $\Mod_{R^\circ}^{\cn} \rightarrow (\Mod_{R^\circ}^{\Cpl(I)})_{\geq 0}$ of \ref{completion is right t-exact} is symmetric monoidal, it determines a functor $\CAlg^{\cn}_{R^\circ} \rightarrow (\CAlg^{\Cpl(I)}_{R^\circ})_{\geq 0}$ which is a left adjoint to the inclusion functor; we will refer to this left adjoint as the \emph{$I$-completion functor}. We will denote the image of $A \in \CAlg^{\cn}_{R^\circ}$ under this left adjoint by $A^\wedge_I$.
\end{notation}

\begin{pg}
	The theory of completions in the derived setting, based on the theory of completions in the spectral setting of \cite{lurie2018sag}, appears in \cite{MR4560539}. In what follows, we will give a functorial definition of completions of simplicial commutative rings and show that it coincides with the definition appearing in \cite[2.1.7]{MR4560539}; see \ref{comparison of completions}. 
\end{pg}

\begin{pg}
	Let $\RPres$ denote the $\infty$-category whose objects are presentable $\infty$-categories and whose morphisms are those functors which are accessible and preserve small limits; see \cite[5.5.3.1]{MR2522659}. Let $I \subseteq \pi_0R$ be a finitely generated ideal. According to \cite[5.5.3.18]{MR2522659}, the $\infty$-category $\CAlg^{\Delta, \Cpl(I)}_R$ fits into a pullback square 
$$
\Pull{\CAlg^{\Delta, \Cpl(I)}_R}{\CAlg^\Delta_R}{(\Mod_{R^\circ}^{\Cpl(I)})_{\geq 0}}{\Mod_{R^\circ}^{\cn}}{}{}{G \circ \Theta}{}
$$
in the $\infty$-category $\RPres$. It follows from the adjoint functor theorem of \cite[5.5.2.9]{MR2522659} that the inclusion functor $\CAlg^{\Delta, \Cpl(I)}_R \rightarrow \CAlg^\Delta_R$ admits a left adjoint. 
\end{pg}

\begin{definition}\label{completion in DAG}
	Let $I \subseteq \pi_0R$ be a finitely generated ideal. We will refer to the left adjoint to the inclusion functor $\CAlg^{\Delta, \Cpl(I)}_R \rightarrow \CAlg^\Delta_R$ as the \emph{$I$-completion functor} and denote the image of $A \in \CAlg^\Delta_R$ under this left adjoint by $A^\wedge_I$.
\end{definition}

\begin{pg}\label{completions of derived rings}
	The composition of the $I$-completion functor $\CAlg^{\Delta, \Cpl(I)}_R \rightarrow \CAlg^\Delta_R$ with $\Theta: \CAlg^\Delta_R \rightarrow \CAlg^{\cn}_{R^\circ}$ factors through the inclusion functor $\CAlg^{\cn, \Cpl(I)}_{R^\circ} \rightarrow \CAlg^{\cn}_{R^\circ}$, so that we have a pullback diagram 
$$
\Pull{\CAlg^{\Delta, \Cpl(I)}_R}{\CAlg^\Delta_R}{\CAlg^{\cn, \Cpl(I)}_{R^\circ}}{\CAlg^{\cn}_{R^\circ}}{}{}{\Theta}{}
$$
in the $\infty$-category $\RPres$. 
\end{pg}

\begin{pg}
	To understand the $I$-completion functor $\CAlg^{\Delta, \Cpl(I)}_R \rightarrow \CAlg^\Delta_R$ of \ref{completion in DAG}, we will need the following analogue of \cite[8.1.2.2]{lurie2018sag} in the derived setting, which can be proven in the same way as \cite[8.1.2.2]{lurie2018sag} (using derived symmetric algebras in place of free $\bbE_\infty$-algebras; see \ref{free algebras}): 
\end{pg}

\begin{lemma}\label{completions as limits in DAG} 
	Let $R$ be an adic simplicial commutative ring with a finitely generated ideal of definition $I \subseteq \pi_0R$. Then there exists a tower $\cdots R_3 \rightarrow R_2 \rightarrow R_1$ in $\CAlg^\Delta_R$ satisfying the following conditions: 
\begin{itemize}
\item[$(1)$] Each of the induced maps $\pi_0R_{i+1} \rightarrow \pi_0R_i$ is a surjection with nilpotent kernel. 
\item[$(2)$] For every simplicial commutative ring $A$, the canonical map $\colim \Map_{\CAlg^\Delta}(R_n, A) \rightarrow \Map_{\CAlg^\Delta}(R, A)$ induces a homotopy equivalence from $\colim \Map_{\CAlg^\Delta}(R_n, A)$ to the summand of $\Map_{\CAlg^\Delta}(R, A)$ spanned by those maps $R \rightarrow A$ which annihilate some power of the ideal $I$. 
\item[$(3)$] Each of the simplicial commutative rings $R_n$ is almost perfect when viewed as an $R^\circ$-module. 
\end{itemize}
\end{lemma}

\begin{remark}\label{tower from DAG to SAG}
	Suppose we are given objects $A \in \CAlg^\Delta_R$ and $B \in \CAlg_{R^\circ}^{\cn}$. Then giving a map $f: \Theta(A) \rightarrow B$ in $\CAlg_{R^\circ}^{\cn}$ is equivalent to giving a map $g: A \rightarrow \Theta^R(B)$ in $\CAlg^\Delta_R$. Let $I \subseteq \pi_0A$ be an ideal. Since the counit map $\Theta(\Theta^R(B)) \rightarrow B$ induces an injection on connected components by virtue of \ref{counit map and injection on connected components}, it follows that $f$ annihilates some power of the ideal $I$ if and only if $g$ annihilates some power of the ideal of $I$. Using this observation, we immediately deduce that in the situation of \ref{completions as limits in DAG}, the induced tower $\{ \Theta(R_n) \}$ of connective $\bbE_\infty$-algebras over $R^\circ$ satisfies the requirements of \cite[8.1.2.2]{lurie2018sag}. 
\end{remark}

\begin{lemma}\label{formal spectrum as a colimit}
	In the situation of \emph{\ref{completions as limits in DAG}}, the functor $F: \CAlg^{\Delta, \Cpl(I)}_R \rightarrow \lim \CAlg^\Delta_{R_n}$, given on objects by the formula $A \mapsto \{A\otimes_R R_n\}$, is an equivalence of $\infty$-categories. 
\end{lemma}

\begin{proof}
	Note that the functor $F$ admits a right adjoint $G: \lim \CAlg^\Delta_{R_n} \rightarrow \CAlg^{\Delta, \Cpl(I)}_R$, given on objects by $\{A_n\} \mapsto \lim A_n$. It will suffice to show that $F$ and $G$ are mutually inverse equivalences. Since the forgetful functors $\CAlg^{\Delta, \Cpl(I)}_R \rightarrow (\Mod_{R^\circ}^{\Cpl(I)})_{\geq 0}$ and $\CAlg^\Delta_{R_n} \rightarrow \Mod_{R_n^\circ}^{\cn}$ are conservative and preserve small limits, we are reduced to proving that the adjoint functors $\adjoint{(\Mod_{R^\circ}^{\Cpl(I)})_{\geq 0}}{\lim \Mod_{R_n^\circ}^{\cn}}$ are mutually inverse equivalences, where the left and right adjoints are given on objects by $M \mapsto \{M\otimes_{R^\circ} R_n^{\circ}\}$ and $\{M_n\} \mapsto \lim M_n$, respectively. This follows from the proof of \cite[8.3.4.4]{lurie2018sag}, since the tower $\{\Theta(R_n) \}$ in $\CAlg_{R^\circ}^{\cn}$ satisfies the requirements of \cite[8.1.2.2]{lurie2018sag} (see \ref{tower from DAG to SAG}). 
\end{proof}

\begin{proposition}\label{comparison of completions}
	The diagram appearing in \emph{\ref{completions of derived rings}} is left adjointable. In other words, for every simplicial commutative $R$-algebra $A$, the canonical map $(\Theta(A))^\wedge_I \rightarrow \Theta(A^\wedge_I)$ is an equivalence in $\CAlg^{\Cpl(I)}_{R^\circ}$. 
\end{proposition}

\begin{proof}
	Let $\{ R_n \}_{n>0}$ be a tower as in \ref{completions as limits in DAG}. It then follows from \ref{formal spectrum as a colimit} that we can identify $A^\wedge_I$ with $\lim A\otimes_R R_n$. According to \ref{tower from DAG to SAG}, the induced tower $\{ \Theta(R_n) \}$ in $\CAlg_{R^\circ}^{\cn}$ satisfies the requirements of \cite[8.1.2.2]{lurie2018sag}, so that we can identify $(\Theta(A))^\wedge_I$ with $\lim \Theta(A)\otimes_{R^\circ} R_n^\circ$; see \cite[8.1.2.3]{lurie2018sag}. The desired result now follows from the fact that the functor $\Theta$ preserves small limits and tensor products; see \ref{the forgetful functor is monadic and comonadic}. 
\end{proof}

\begin{remark}\label{complete derived rings are compatible}
	Let $A$ be an adic simplicial commutative $R$-algebra with a finitely generated ideal of definition $I \subseteq \pi_0A$. It follows from \ref{comparison of completions} and \ref{the forgetful functor is monadic and comonadic} (which guarantees that $\Theta$ is conservative) that $A$ is an $I$-complete simplicial commutative $R$-algebra if and only if $\Theta(A)$ is an $I$-complete connective $\bbE_\infty$-algebra over $R^\circ$ (see \ref{complete rings in the derived setting} and \ref{complete rings in the spectral setting}). In particular, $R$ is $I$-complete if and only if the underlying connective $\bbE_\infty$-ring $R^\circ$ is $I$-complete. 
\end{remark}

\begin{pg}
	We now define the \emph{formal spectrum} of a simplicial commutative ring by proceeding as in \cite[8.1.1.10]{lurie2018sag}.
\end{pg}

\begin{definition}\label{structure sheaf for formal spectrum} 
	Let $A$ be an adic simplicial commutative ring and let $I \subseteq \pi_0A$ be a finitely generated ideal of definition; see \ref{adic simplicial commutative rings}. Let $U: \CAlg^{\Delta, \et}_A \rightarrow \SSet$ denote the sheaf with respect to the \'etale topology of \ref{etale topology on derived rings} given by the formula
$$
U(B)=
\begin{cases}
\Delta^0 & \text{if } B\otimes_A (\pi_0A)/I \simeq 0 \\
\emptyset & \text{otherwise}.
\end{cases}
$$
Let $\Shv^{\Delta, \ad}_A \subseteq \Shv((\CAlg^{\Delta, \et}_A)^{\op})$ denote the \emph{closed subtopos complementary to $U$} in the sense of \cite[7.3.2.6]{MR2522659}: that is, it is the full subcategory spanned by those objects $F \in \Shv((\CAlg^{\Delta, \et}_A)^{\op})$ for which $F(B)$ is contractible whenever the image of $I$ in $\pi_0B$ generates the unit ideal. We let $\calO_{\Spf A}$ denote the composition of the forgetful functor $\CAlg^{\Delta, \et}_A \rightarrow \CAlg^\Delta_A$ with the $I$-completion functor $\CAlg^\Delta_A \rightarrow \CAlg^{\Delta, \Cpl(I)}_A \subseteq \CAlg^\Delta_A$ of \ref{completion in DAG}. 
\end{definition}

\begin{pg}
	Since the forgetful functor $\CAlg^{\Delta, \et}_A \rightarrow \CAlg^\Delta_A$ is a $\CAlg^\Delta_A$-valued sheaf with respect to the \'etale topology on $(\CAlg^{\Delta, \et}_A)^{\op}$ and the $I$-completion functor of \ref{completion in DAG} preserves small limits, we can regard $\calO_{\Spf A}$ as a $\CAlg^\Delta_A$-valued sheaf on the $\infty$-topos $\Shv((\CAlg^{\Delta, \et}_A)^{\op})$. Moreover, $\calO_{\Spf A}$ can be regarded as a $\CAlg^\Delta_A$-valued sheaf on the closed subtopos $\Shv^{\Delta, \ad}_A$, since $\calO_{\Spf A}(B)=B^\wedge_I \simeq 0$ whenever $I(\pi_0B)=\pi_0B$
\end{pg}

\begin{pg}
	We will need the following analogue of \cite[8.1.1.13]{lurie2018sag} in the derived setting:
\end{pg}

\begin{lemma}\label{structure sheaves for affine formal spectra are strictly Henselian}
	Let $A$ be an adic simplicial commutative ring with a finitely generated ideal of definition $I \subseteq \pi_0A$. Then the $\CAlg^\Delta_A$-valued sheaf $\calO_{\Spf A}$ on the $\infty$-topos $\Shv^{\Delta, \ad}_A$ is strictly Henselian in the sense of \emph{\ref{strictly Henselian objects and local morphisms in the derived setting}}. 
\end{lemma}

\begin{proof}
	Let us regard $A^\circ$ as an adic $\bbE_\infty$-ring with the ideal of definition $I$ (see \ref{adic derived rings}). We first note that the closed subtopos $\Shv^{\Delta, \ad}_A \subseteq \Shv((\CAlg^{\Delta, \et}_A)^{\op})$ can be identified with the closed subtopos $\Shv^{\ad}_{A^\circ} \subseteq \Shv((\CAlg^{\et}_{A^\circ})^{\op})$ of \cite[8.1.1.8]{lurie2018sag} under the equivalence of $\infty$-categories $\Shv((\CAlg^{\Delta, \et}_A)^{\op}) \simeq \Shv((\CAlg^{\et}_{A^\circ})^{\op})$ of \ref{the equivalence of etale sites}. Let $\calO_{\Spf A^\circ}$ be the composition of the forgetful functor $\CAlg^{\et}_{A^\circ} \rightarrow \CAlg_{A^\circ}^{\cn}$ with the $I$-completion functor $\CAlg_{A^\circ}^{\cn} \rightarrow (\CAlg^{\Cpl(I)}_{R^\circ})_{\geq 0} \subseteq \CAlg_{A^\circ}^{\cn}$ of \ref{complete rings in the spectral setting}. Using \cite[8.1.1.10]{lurie2018sag}, we can regard $\calO_{\Spf A^\circ}$ as a $\CAlg_{A^\circ}^{\cn}$-valued sheaf on $\Shv^{\ad}_{A^\circ}$. Then \cite[8.1.1.13]{lurie2018sag} guarantees that $\calO_{\Spf A^\circ}$ is strictly Henselian. Consequently, to prove that $\calO_{\Spf A}$ is strictly Henselian, it will suffice to show that $\Theta_A \circ \calO_{\Spf A}$ can be identified with $\calO_{\Spf A^\circ}$ under the equivalence of $\infty$-topos $\Shv^{\Delta, \ad}_A \simeq \Shv^{\ad}_{A^\circ}$; see \ref{strictly Henselian objects and local morphisms in the spectral setting}. Unwinding the definitions of $\calO_{\Spf A}$ and $\calO_{\Spf A^\circ}$, we are reduced to proving that the diagram
$$
\xymatrix{
\CAlg^\Delta_A \ar[r] \ar[d]^-{\Theta_A} & \CAlg^{\Delta, \Cpl(I)}_A \ar[d] \\
\CAlg_{A^\circ}^{\cn} \ar[r] & \CAlg^{\cn, \Cpl(I)}_{A^\circ}
}
$$
commutes, where the horizontal maps are the $I$-completion functors and the right vertical map is the restriction of $\Theta_A$ appearing in \ref{completions of derived rings}. This follows from \ref{comparison of completions}. 
\end{proof}

\begin{notation} 
	Let $A$ be an adic simplicial commutative ring with a finitely generated ideal of definition $I \subseteq \pi_0A$. It follows from \ref{structure sheaves for affine formal spectra are strictly Henselian} that we can regard the pair $(\Shv^{\Delta, \ad}_A,\calO_{\Spf A})$ as an object of $\iTop_{\CAlg^\Delta_A}^{\sHen}$. We will denote this object by $\Spf A$ and refer to it as the \emph{formal spectrum of $A$}.  
\end{notation}

\begin{remark}\label{the canonical map from the formal spectra} 
	The inclusion of $\infty$-topoi $\Shv^{\Delta, \ad}_A \subseteq \Shv((\CAlg^{\Delta, \et}_A)^{\op})$ and the $I$-completion functor of \ref{completion in DAG} determine a morphism $\Spf A \rightarrow \Spec A$ in $\iTop_{\CAlg^\Delta_A}^{\sHen}$ .
\end{remark}

\begin{definition} 
	Let $\frakX=(\calX, \calO_{\calX})$ be an object of the $\infty$-category $\iTop_{\CAlg^\Delta}$. We will say that $\frakX$ is an \emph{affine formal derived Deligne-Mumford stack} if it is equivalent to $\Spf A$ for some adic simplicial commutative ring $A$. More generally, we will say that $\frakX$ is a \emph{formal derived Deligne-Mumford stack} if there exists a collection of objects $U_\alpha \in \calX$ for which the map $\coprod U_\alpha \rightarrow \ast$ is an effective epimorphism, where $\ast \in \calX$ is a final object, and each $\frakX_{U_\alpha}$ is an affine formal derived Deligne-Mumford stack (see \ref{sheaves on infinity-topoi}). 
\end{definition}

\begin{notation}\label{the infinity-categories of formal derived and spectral Deligne-Mumford stacks}
	We let $\fDerDM_R \subseteq \iTop_{\CAlg^\Delta_R}^{\sHen}$ and $\fSpDM_{R^\circ} \subseteq \iTop_{\CAlg_{R^\circ}^{\cn}}^{\sHen}$ denote the full subcategories spanned by the formal derived and spectral Deligne-Mumford stacks over $\Spec R$ and $\Spec R^\circ$, respectively (see \cite[8.1.3.1]{lurie2018sag}). 
\end{notation}

\begin{remark}\label{derived Deligne-Mumford stacks as formal derived Deligne-Mumford stacks}
	Every derived Deligne-Mumford stack can be regarded as a formal derived Deligne-Mumford stack, since we can regard each simplicial commutative ring $A$ as an adic simplicial commutative ring by taking the zero ideal of $\pi_0A$ as an ideal of definition. 
\end{remark}

\begin{proposition}\label{the underlying forma spectra}
	Let $A$ be an adic simplicial commutative $R$-algebra with a finitely generated ideal of definition $I \subseteq \pi_0A$. Then the canonical map $\Xi(\Spf A) \rightarrow \Spf \Theta(A)$ is an equivalence in $\iTop_{\CAlg_{R^\circ}^{\cn}}^{\sHen}$, where we regard $\Theta(A)$ as an adic $\bbE_\infty$-algebra over $R^\circ$ with the ideal of definition $I$ (see \emph{\ref{adic derived rings}}). Moreover, if $\frakX$ is a formal derived Deligne-Mumford stack over $R$, then $\Xi(\frakX)$ is a formal spectral Deligne-Mumford stack over $R^\circ$. 
\end{proposition}

\begin{proof}
	The first assertion is immediate from the proof of \ref{structure sheaves for affine formal spectra are strictly Henselian}. The second follows from the first assertion, since the condition that $\Xi(\frakX)$ be a formal spectral Deligne-Mumford stack over $R^\circ$ is local on the underlying $\infty$-topos of $\frakX$ by virtue of the derived analogue of \cite[8.1.3.3]{lurie2018sag} (see also the proof \ref{associated spectral and derived Deligne-Mumford stacks}).
\end{proof}

\begin{remark}\label{the functor for formal derived and spectral Deligne-Mumford stacks}
	It follows from \ref{the underlying forma spectra} that the functor $\Xi: \iTop_{\CAlg^\Delta_R}^{\sHen} \rightarrow \iTop_{\CAlg^{\cn}_{R^\circ}}^{\sHen}$ of \ref{restriction to the subcategory of strictly Henselian objects and local morphisms} restricts to a functor $\fDerDM_R \rightarrow \fSpDM_{R^\circ}$. 
\end{remark}

\begin{remark}\label{affine formal spectra and completions}
	Let $A$ be an adic simplicial commutative $R$-algebra with a finitely generated ideal of definition $I \subseteq \pi_0A$. Let us regard $A^\wedge_I$ as an adic simplicial commutative $R$-algebra with the ideal of definition $I(\pi_0 A^\wedge_I)$. Then the induced map $\Spf A^\wedge_I \rightarrow \Spf A$ is an equivalence in $\fDerDM_R$ by virtue of \cite[8.1.2.4]{lurie2018sag} and \ref{comparison of completions} (see also \ref{conservative at the level of ringed infinity-topoi} and \ref{the underlying forma spectra}).
\end{remark}

\begin{remark}\label{affine formal derived stacks as a colimit of affine derived stacks} 
	Let $A$ be an adic simplicial commutative $R$-algebra and let $I \subseteq \pi_0A$ be a finitely generated ideal of definition. Let $\{A_n \}$ be a tower of simplicial commutative $A$-algebras which satisfies the requirements of \ref{completions as limits in DAG}, so that the tower $\{ \Theta(A_n)\}$ of connective $\bbE_\infty$-algebras over $R^\circ$ satisfies the requirements of \cite[8.1.2.2]{lurie2018sag} by virtue of \ref{tower from DAG to SAG}. According to \cite[8.1.2.1]{lurie2018sag}, the canonical map $\colim \Spec \Theta(A_n) \rightarrow \Spf \Theta(A)$ is an equivalence in $\iTop_{\CAlg^{\cn}_{R^\circ}}^{\sHen}$. Since $\Xi$ is conservative (see \ref{conservative at the level of ringed infinity-topoi}), it follows from \ref{the underlying forma spectra} that the canonical map $\colim \Spec A_n \rightarrow \Spf A$ is an equivalence in $\iTop_{\CAlg^\Delta_R}^{\sHen}$. 
\end{remark}

\begin{definition}\label{underlying formal spectral Deligne-Mumford stacks}
	Let $\frakX$ be a formal derived Deligne-Mumford stack over $R$. Then $\Xi(\frakX)$ is a formal spectral Deligne-Mumford stack over $R^\circ$ by virtue of \ref{the underlying forma spectra}. We will refer to $\Xi(\frakX)$ as the \emph{underlying formal spectral Deligne-Mumford stack of $\frakX$ over $R^\circ$}.
\end{definition}

\begin{definition}\label{quasi-coherent sheaves in the formal setting}
	Let $\frakX$ be a formal derived Deligne-Mumford stack. Let $\QCoh(\frakX)$ denote the \emph{$\infty$-category of quasi-coherent sheaves on the underlying formal spectral Deligne-Mumford stack $\Xi_\mathbb{Z}(\frakX)$} of \ref{underlying formal spectral Deligne-Mumford stacks}, in the sense of \cite[8.2.4.7]{lurie2018sag}. We let $\APerf(\frakX) \subseteq \QCoh(\frakX)$ denote the full subcategory spanned by those quasi-coherent sheaves which are almost perfect in the sense of \cite[8.3.5.1]{lurie2018sag}.
\end{definition}

\begin{pg}
	We record the following analogue of \cite[8.1.5.2]{lurie2018sag}: 
\end{pg}

\begin{lemma}\label{formal spectra via the functor of points}
	Let $A$ be an adic simplicial commutative ring with a finitely generated ideal of definition $I \subseteq \pi_0A$ and let $B$ be a simplicial commutative $R$-algebra. Then the map
$$
\Map_{\iTop_{\CAlg^\Delta_R}^{\sHen}}(\Spec B, \Spf A) \rightarrow \Map_{\iTop_{\CAlg^\Delta_R}^{\sHen}}(\Spec B, \Spec A),
$$
given by composition with the map $\Spf A \rightarrow \Spec A$ of \emph{\ref{the canonical map from the formal spectra}}, induces a homotopy equivalence of $\Map_{\iTop_{\CAlg^\Delta_R}^{\sHen}}(\Spec B, \Spf A)$ with the summand of $\Map_{\iTop_{\CAlg^\Delta_R}^{\sHen}}(\Spec B, \Spec A) \simeq \Map_{\CAlg^\Delta}(A, B)$ spanned by those maps $\phi: A \rightarrow B$ which annihilate some power of $I$. 
\end{lemma}

\begin{proof}
	This assertion can be proven in exactly the same way as \cite[8.1.5.2]{lurie2018sag}, using the following argument in place of \cite[8.1.1.14]{lurie2018sag}: suppose we are given a map of simplicial commutative $R$-algebras $\phi: A \rightarrow B$ which annihilates some power of the ideal $I \subseteq \pi_0A$. Let $f: \Spec B \rightarrow \Spec A$ denote the induced map of derived Deligne-Mumford stacks over $R$. Then the underlying morphism of $\infty$-topoi $\Shv((\CAlg^{\Delta, \et}_B)^{\op}) \rightarrow \Shv((\CAlg^{\Delta, \et}_A)^{\op})$ factors through the closed subtopos $\Shv^{\Delta, \ad}_A \subseteq \Shv((\CAlg^{\Delta, \et}_A)^{\op})$. If $\overline{f}: \Spec B \rightarrow \Spf A$ is a lift of $f$ in $\iTop_{\CAlg^\Delta_R}$, it then follows from the proof of \cite[8.1.5.2]{lurie2018sag} and \ref{the underlying forma spectra} that $\Xi(\overline{f})$ is a local map, so that $\overline{f}$ is also a local map. 
\end{proof}

\begin{notation}
	Let $\frakX$ be a formal derived Deligne-Mumford stack over $R$. We let $h_{\frakX}: \CAlg^\Delta_R \rightarrow \widehat{\SSet}$ denote the functor represented by $\frakX$, given on objects by the formula $h_{\frakX}(A)=\Map_{\iTop_{\CAlg^\Delta_R}^{\sHen}}(\Spec A, \frakX)$, where $\widehat{\SSet}$ denotes the $\infty$-category of (not necessarily small) spaces of \cite[1.2.16.4]{MR2522659}. 
\end{notation}

\begin{pg}
	We will need the following derived analogue of \cite[8.1.5.1]{lurie2018sag}, which can be proven in the same way (using \ref{affine formal derived stacks as a colimit of affine derived stacks} and \ref{formal spectra via the functor of points} in place of \cite[8.1.2.1]{lurie2018sag} and \cite[8.1.5.2]{lurie2018sag}): 
\end{pg}

\begin{lemma}\label{functor of points for formal derived stacks}
	Let $\frakX$ be a formal derived Deligne-Mumford stack over $R$. Then the space $h_{\frakX}(A)$ is essentially small for each simplicial commutative $R$-algebra $A$. Moreover, the construction $\frakX \mapsto h_{\frakX}$ determines a fully faithful embedding $\fDerDM_R \rightarrow \Fun(\CAlg^\Delta_R, \SSet)$. 
\end{lemma}

\begin{pg}
	According to \cite[8.1.7.1]{lurie2018sag} and its derived analogue, both the $\infty$-categories $\fDerDM_R$ and $\fSpDM_{R^\circ}$ admit finite limits. Moreover, we have the following result:
\end{pg}

\begin{lemma}\label{preservation of fiber products in the formal case}
	The functor $\Xi: \fDerDM_R \rightarrow \fSpDM_{R^\circ}$ preserves finite limits.
\end{lemma}

\begin{proof}
	The functor $\Xi$ carries the final object $\Spec R \in \fDerDM_R$ to the final object $\Spec R^\circ \in \fSpDM_{R^\circ}$; see \ref{derived Deligne-Mumford stacks as formal derived Deligne-Mumford stacks}. It will therefore suffice to show that $\Xi$ preserves fiber products. Arguing as in the proof of \ref{preservation of fiber products}, we are reduced to proving that the canonical map $\delta: \Xi(\Spf A_2 \times_{\Spf A_1} \Spf A_3) \rightarrow \Xi(\Spf A_2) \times_{\Xi(\Spf A_1)} \Xi(\Spf A_3)$ is an equivalence, where each $A_i$ is an adic simplicial commutative $R$-algebra with a finitely generated ideal of definition $I_i \subseteq \pi_0A_i$. Using \ref{affine formal spectra and completions}, we can replace $A_2$ and $A_3$ by their completions and thereby reduce to the case where $A_2$ and $A_3$ are complete. In this case, the analogue of \cite[8.1.5.4]{lurie2018sag} in the derived setting guarantees that for each $i \in \{2,3\}$, the morphism $\Spf A_i \rightarrow \Spf A_1$ is induced by a map of adic simplicial commutative $R$-algebras $A_1 \rightarrow A_i$ (that is, a map of simplicial commutative $R$-algebras which annihilates some power of the ideal $I_1$). Let us regard $A_2 \otimes_{A_1} A_3$ as an adic simplicial commutative $R$-algebra by equipping $\pi_0(A_2 \otimes_{A_1} A_3)$ with the $I$-adic topology, where $I$ is the ideal of $\pi_0(A_2 \otimes_{A_1} A_3)$ generated by the images of $I_2$ and $I_3$. It then follows from \cite[8.1.7.3]{lurie2018sag} and its derived analogue that the domain and codomain of $\delta$ can be identified with $\Xi(\Spf A_2 \otimes_{A_1} A_3)$ and $\Spf \Theta(A_2) \otimes_{\Theta(A_1)} \Theta(A_3)$, respectively. The desired result now follows from \ref{the underlying forma spectra} and \ref{the forgetful functor is monadic and comonadic} (which guarantees that $\Theta$ preserves tensor products).
\end{proof}

\begin{pg}
	We now introduce the notion of \emph{formal completions} in the derived setting:
\end{pg}

\begin{definition}\label{formal completion in DAG}
	Let $R$ be an adic simplicial commutative ring with a finitely generated ideal of definition $I \subseteq \pi_0R$. Let $\sfX$ be a derived Deligne-Mumford stack over $R$. We will denote by $\sfX^\wedge_I$ the fiber product $\sfX \times_{\Spec R} \Spf R$, formed in the $\infty$-category $\fDerDM$ (see \ref{derived Deligne-Mumford stacks as formal derived Deligne-Mumford stacks}). We will refer to $\sfX^\wedge_I$ as the \emph{formal completion of $\sfX$ along the vanishing locus of $I$}. 
\end{definition}

\begin{remark}\label{comparison for completions}
	According to \cite[8.1.6.1]{lurie2018sag}, the \emph{formal completion $\Xi(\sfX)^\wedge_I$ of the spectral Deligne-Mumford stack $\Xi(\sfX)$ along the vanishing locus of $I$} (regarded as an ideal of $\pi_0 R^\circ$) is defined to be the fiber product $\sfX \times_{\Spec R^\circ} \Spf R^\circ$ formed in the $\infty$-category $\fSpDM$; here we regard $R^\circ$ as an adic $\bbE_\infty$-ring with the ideal of definition $I$ by virtue of \ref{adic derived rings}. Note that \ref{preservation of fiber products in the formal case} supplies a canonical equivalence $\Xi(\sfX^\wedge_I) \rightarrow \Xi(\sfX)^\wedge_I$ in $\fSpDM$.
\end{remark}

\begin{pg}
	We conclude this section by describing the behavior of formal completions with respect to colimits.
\end{pg}

\begin{lemma}\label{completions and colimits}
	Let $R$ be an adic simplicial commutative ring with a finitely generated ideal of definition $I \subseteq \pi_0R$. Suppose we are given a collection of morphisms $\{\sfX_\alpha \rightarrow \sfX \}$ which exhibits $\sfX$ as a colimit of the family $\{ \sfX_\alpha \}$ in the $\infty$-category $\DerDM_R$. Then the induced diagram $\{(\sfX_\alpha)^\wedge_I \rightarrow \sfX^\wedge_I \}$ exhibits $\sfX^\wedge_I$ as a colimit of the family $\{(\sfX_\alpha)^\wedge_I \}$ in the $\infty$-category $\fDerDM_R$.
\end{lemma}

\begin{proof}
	Since the functor represented by a formal derived Deligne-Mumford stack satisfies descent for the \'etale topology (see \ref{functor of points for formal derived stacks}), the desired result follows from \cite[6.1.0.6]{MR2522659}, which guarantees that colimits are universal in the $\infty$-topos $\Shv((\CAlg^{\Delta,\et}_R)^{\op})$ of \ref{etale topology on derived rings}. 
\end{proof}

\section{Formal GAGA in the Derived Setting} 

\begin{pg}
	The main goal of this section is to establish a version of the formal GAGA theorem of \cite[8.5.3.1]{lurie2018sag} in the derived setting. 
\end{pg}

\begin{pg}
	Let $\sfX=(\calX, \calO_{\calX}) \in \iTop^{\sHen}_{\CAlg}$ be an object and let $n \geq 0$ be an integer. We will say that $\sfX$ is \emph{$n$-truncated} if $\Psi_{\calX}(\calO_{\calX})$ is $n$-truncated in the sense of \cite[1.3.5.5]{lurie2018sag}; see \ref{comparison for sheaves on derived rings}. We let $(\iTop^{\sHen}_{\CAlg})^{\leq n} \subseteq \iTop^{\sHen}_{\CAlg}$ denote the full subcategory spanned by the $n$-truncated objects. We have the following analogue of \cite[1.4.6.3]{lurie2018sag}, which can be proven by exactly the same argument:
\end{pg}

\begin{lemma}\label{truncation as right adjoint}
	Let $n \geq 0$ be an integer. Then the inclusion functor $(\iTop^{\sHen}_{\CAlg})^{\leq n} \rightarrow \iTop^{\sHen}_{\CAlg}$ admits a right adjoint, given on objects by $(\calX, \calO_{\calX}) \mapsto (\calX, \tau_{\leq n}\calO_{\calX})$. 
\end{lemma}

\begin{notation}
	Let $\sfX=(\calX, \calO_{\calX})$ be an object of $\iTop^{\sHen}_{\CAlg}$ and let $n \geq 0$ be an integer. We will denote $(\calX, \tau_{\leq n}\calO_{\calX})$ by $\tau_{\leq n} \sfX$ and refer to it as the \emph{$n$-truncation of $\sfX$}.
\end{notation}

\begin{remark}\label{equivalence between 0-truncated fDerDM and fSpDM}
	Using \ref{equivalence of 0-truncations of derived rings}, we see that there is an equivalence of $\infty$-categories between the full subcategory of $\fDerDM_R$ spanned by the $0$-truncated formal derived Deligne-Mumford stacks and the full subcategory of $\fSpDM_{R^\circ}$ spanned by the $0$-truncated formal spectral Deligne-Mumford stacks.
\end{remark}

\begin{pg}
	The following definitions are derived analogues of \cite[2.8.1.4]{lurie2018sag} and \cite[4.2.0.1]{lurie2018sag}, respectively:
\end{pg}

\begin{definition}\label{locally noetherian derived Deligne-Mumford stacks}
	Let $\sfX$ be a derived Deligne-Mumford stack. We will say that $\sfX$ is \emph{locally noetherian} if, for every \'etale morphism $\Spec A \rightarrow \sfX$, the simplicial commutative ring $A$ is noetherian (that is, $\pi_0A$ is an ordinary noetherian ring and each $\pi_nA$ is finitely generated as a $\pi_0A$-module). 
\end{definition}

\begin{definition}\label{morphisms of locally almost of finite presentation}
	Let $f: \sfX \rightarrow \sfY$ be a morphism of derived Deligne-Mumford stacks. We will say that $f$ is \emph{locally almost of finite presentation} if, for every commutative diagram of derived Deligne-Mumford stacks
$$
\xymatrix{
\Spec B \ar[r] \ar[d] & \sfX \ar[d]^-f \\
\Spec A \ar[r] & \sfY
}
$$
where the horizontal maps are \'etale, the simplicial commutative ring $B$ is almost of finite presentation over $A$ (that is, $\tau_{\leq n}B$ is compact as an object of $\tau_{\leq n}\CAlg^\Delta_A$ for every $n \geq 0$). 
\end{definition}

\begin{pg}
	We will need the following assertions, which are usually not true if we do not assume that $\sfX$ is locally noetherian.
\end{pg}

\begin{lemma}\label{truncation and almost of finite presentation} 
	Let $\sfX=(\calX, \calO_{\calX})$ be a locally noetherian derived Deligne-Mumford stack, in the sense of \emph{\ref{locally noetherian derived Deligne-Mumford stacks}}. Then for every integer $n \geq 0$, the canonical map $\tau_{\leq n} \sfX \rightarrow \sfX$ is locally almost of finite presentation in the sense of \emph{\ref{morphisms of locally almost of finite presentation}} and $\pi_{n+1} \calO_{\calX}$ is almost perfect as an object of $\QCoh(\tau_{\leq n} \sfX)$ of \emph{\ref{quasi-coherent sheaves in the formal setting}}.
\end{lemma}

\begin{proof}
	The assertion is local on $\sfX$; we may therefore suppose that $\sfX=\Spec A$ is affine. Our assumption on $\sfX$ guarantees that $A$ is a noetherian simplicial commutative ring. Since $\tau_{\leq n}A$ is also noetherian and $\pi_{n+1}A$ is a finitely generated module over $\pi_0A$, the desired result follows from derived analogues of \cite[7.2.4.31]{lurie2017ha} and \cite[7.2.4.17]{lurie2017ha}. 
\end{proof}

\begin{lemma}\label{completions and truncations}
	Let $R$ be an adic simplicial commutative ring with a finitely generated ideal of definition $I \subseteq \pi_0R$. Let $\sfX$ be a locally noetherian derived Deligne-Mumford stack over $R$. Then the canonical map $\tau_{\leq n} (\sfX^\wedge_I) \rightarrow (\tau_{\leq n} \sfX)^\wedge_I$ is an equivalence in the $\infty$-category $\iTop^{\sHen}_{\CAlg}$. 
\end{lemma}

\begin{proof}
	Working locally on $\sfX$, we can reduce to the case where $\sfX=\Spec A$ is affine. We wish to show that $\calO_{\Spf \tau_{\leq n}A}$ can be identified with $\tau_{\leq n}\calO_{\Spf A}$ under the equivalence of $\infty$-topoi $\Shv^{\Delta, \ad}_{\tau_{\leq 0}A} \simeq \Shv^{\Delta, \ad}_A$; see \ref{structure sheaf for formal spectrum}. Unwinding the definitions, we are reduced to proving that for each simplicial commutative ring $B$ which is \'etale over $A$, the canonical map $(\tau_{\leq n}B)^\wedge_{I(\pi_0A)} \rightarrow \tau_{\leq n}(B^\wedge_{I(\pi_0A)})$ is an equivalence of simplicial commutative rings. Since the $I(\pi_0A)$-completion functor $\Mod_{A^\circ} \rightarrow \Mod_{A^\circ}^{\Cpl(I(\pi_0A))}$ of \ref{completion is right t-exact} is right t-exact by virtue of \cite[7.3.4.4]{lurie2018sag}, it will suffice to show that $(\tau_{\leq n}B)^\wedge_{I(\pi_0A)}$ is $n$-truncated. Writing $B$ as a filtered colimit of perfect $A^\circ$-modules $B_\alpha$ (see \cite[7.2.4.2]{lurie2017ha}), we are reduced to proving that each $(\tau_{\leq n}B_\alpha)^\wedge_{I(\pi_0A)}$ is $n$-truncated. Since $A$ is noetherian, \cite[8.4.4.1]{lurie2018sag} guarantees that each $\tau_{\leq n}B_\alpha$ is almost perfect as an $A^\circ$-module, so that $(\tau_{\leq n}B_\alpha)^\wedge_{I(\pi_0A)}$ can be identified with the tensor product $\tau_{\leq n}B_\alpha \otimes_A A^\wedge_{I(\pi_0A)}$ by virtue of \cite[7.3.5.7]{lurie2018sag}. The desired result now follows from the fact that the canonical map $A \rightarrow A^\wedge_{I(\pi_0A)}$ is flat (see \cite[7.3.6.9]{lurie2018sag} and \ref{comparison for completions}). 
\end{proof}

\begin{pg}\label{a relative version of the formal GAGA}
	Let us recall the formal GAGA in the context of underlying formal spectral Deligne-Mumford stacks of formal derived Deligne-Mumford stacks (see \ref{underlying formal spectral Deligne-Mumford stacks}). To this end, let $R$ be a complete adic simplicial commutative ring with a finitely generated ideal of definition $I \subseteq \pi_0R$. Let $\sfX$ be a derived algebraic space which is proper and locally almost of finite presentation over $R$ and let $\sfX^\wedge_I$ denote the formal completion of $\sfX$ along the vanishing locus of $I$ (see \ref{formal completion in DAG}). Let $\sfY$ be a quasi-separated derived algebraic space. Note that $\Xi_\mathbb{Z}(\sfX)$ is a spectral algebraic space which is proper and locally almost of finite presentation over $R^\circ$ and $\Xi_\mathbb{Z}(\sfY)$ is a quasi-separated spectral algebraic space over $\mathbb{Z}$; see \ref{preservation of properties} and \cite[p.1683]{lurie2018sag}. Moreover, we can regard $R^\circ$ as a complete adic $\bbE_\infty$-ring with the ideal of definition $I$ (regarded as an ideal of $\pi_0R^\circ$) by virtue of \ref{complete derived rings are compatible}. According to \cite[8.5.3.1]{lurie2018sag}, 
the restriction map 
$$
\Map_{\SpDM_\mathbb{Z}}(\Xi_\mathbb{Z}(\sfX), \Xi_\mathbb{Z}(\sfY)) \rightarrow \Map_{\fSpDM_\mathbb{Z}}(\Xi_\mathbb{Z}(\sfX)^\wedge_I, \Xi_\mathbb{Z}(\sfY))
$$
is a homotopy equivalence. Note that $\Xi_\mathbb{Z}(\sfX)^\wedge_I$ can be identified with $\Xi_\mathbb{Z}(\sfX^\wedge_I)$ by virtue of \ref{comparison for completions}. 
\end{pg}

\begin{pg}
	 We are now ready to give a proof of \ref{formal GAGA in DAG}:
\end{pg}

\begin{proof}[Proof of \emph{\ref{formal GAGA in DAG}}]
	We proceed as in the proof of \cite[5.1.13]{MR4560539}. We note that the canonical map $\mathop{\colim}\limits_{n \geq 0} \tau_{\leq n}\sfX \rightarrow \sfX$ is an equivalence in $\DerDM$. It follows from \ref{completions and colimits} that the induced map $\mathop{\colim}\limits_{n \geq 0}(\tau_{\leq n}\sfX)^\wedge_I \rightarrow \sfX^\wedge_I$ is an equivalence in the $\infty$-category $\iTop^{\sHen}_{\CAlg}$. We can therefore reduce to the case where $\sfX$ is $n$-truncated for some integer $n \geq 0$. We now proceed by induction on $n$. Suppose first that $n >0$. Let $\sfX=(\calX, \calO_{\calX})$, and let $(\tau_{\leq n-1}\sfX)^{\pi_n\calO_{\calX}[n+1]}$ denote the derived Deligne-Mumford stack $(\calX, \tau_{\leq n-1}\calO_{\calX} \oplus \pi_n\calO_{\calX}[n+1])$, where $\tau_{\leq n-1}\calO_{\calX} \oplus \pi_n\calO_{\calX}[n+1]$ is the trivial square-zero extension of $\tau_{\leq n-1}\calO_{\sfX}$ by $\pi_n\calO_{\calX}[n+1]$; see \cite[17.1.1]{lurie2018sag}. It follows from \cite[7.4.1.26]{lurie2017ha} that $\tau_{\leq n}\calO_{\calX}$ is a square-zero extension of $\tau_{\leq n-1}\calO_{\calX}$ by $\pi_n\calO_{\calX}[n]$, so that we have a pushout diagram of derived Deligne-Mumford stacks 
$$
\xymatrix{
(\tau_{\leq n-1}\sfX)^{\pi_n\calO_{\calX}[n+1]} \ar[r] \ar[d] & \tau_{\leq n-1}\sfX \ar[d] \\
\tau_{\leq n-1}\sfX \ar[r] & \sfX.
}
$$
We deduce that the restriction map $\theta_{\sfX}$ can be identified with the fiber product of $\theta_{\tau_{\leq n-1}\sfX}$ with itself over $\theta_{(\tau_{\leq n-1}\sfX)^{\pi_n\calO_{\calX}[n+1]}}$ in the $\infty$-category $\Fun(\Delta^1, \SSet)$. Our inductive hypothesis guarantees that the map $\theta_{\tau_{\leq n-1}\sfX}$ is an equivalence. Consequently, to show that $\theta_{\sfX}$ is an equivalence, it will suffice to show that $\theta_{(\tau_{\leq n-1}\sfX)^{\pi_n\calO_{\calX}[n+1]}}$ is an equivalence. We have a commutative diagram
$$
\xymatrix{
\Map_{\DerDM}((\tau_{\leq n-1}\sfX)^{\pi_n\calO_{\calX}[n+1]}, \sfY) \ar[r] \ar[d] & \Map_{\fDerDM}(((\tau_{\leq n-1}\sfX)^{\pi_n\calO_{\calX}[n+1]})^\wedge_I, \sfY) \ar[d] \\
\Map_{\DerDM}(\tau_{\leq n-1}\sfX, \sfY) \ar[r] & \Map_{\fDerDM}((\tau_{\leq n-1} \sfX)^\wedge_I, \sfY),
}
$$
where the bottom horizontal map is an equivalence by the inductive hypothesis. It will therefore suffice to show that this diagram induces a homotopy equivalence between the homotopy fibers of the vertical maps over any point $f: \tau_{\leq n-1}\sfX \rightarrow \sfY$. Let $i:(\tau_{\leq n-1}\sfX)^\wedge_I \rightarrow \tau_{\leq n-1}\sfX$ be a map which exhibits $(\tau_{\leq n-1}\sfX)^\wedge_I$ as the formal completion of $\tau_{\leq n-1}\sfX$ along the vanishing locus of $I$. Unwinding the definition of trivial square-zero extensions, we are reduced to proving that the canonical map
$$
\Map_{\QCoh(\tau_{\leq n-1}\sfX)}(f^\ast L^{\alg}_{\sfY}, \pi_n\calO_{\calX}[n+1]) \rightarrow \Map_{\QCoh((\tau_{\leq n-1}\sfX)^\wedge_I)}(\widehat{f}^\ast L^{\alg}_{\sfY}, i^\ast\pi_n\calO_{\calX}[n+1])
$$
is an equivalence, where $L^{\alg}_{\sfY}$ is the algebraic cotangent complex of $\sfY$ of \cite[25.3.1.6]{lurie2018sag} and $\widehat{f}$ denotes the composition of $f$ with the map $i$; see \ref{quasi-coherent sheaves in the derived setting} and \ref{quasi-coherent sheaves in the formal setting}. Using a derived analogue of \cite[4.2.0.4]{lurie2018sag}, we see that $\sfX$ is locally noetherian, so that $\tau_{\leq n-1}\sfX$ is proper and locally almost of finite presentation over $R$ by virtue of \ref{truncation and almost of finite presentation}. Combining this observation with \ref{a relative version of the formal GAGA} and \cite[8.5.0.3]{lurie2018sag}, we obtain an equivalence of $\infty$-categories $i^\ast: \APerf((\tau_{\leq n}\sfX)^\wedge_I) \rightarrow \APerf(\tau_{\leq n}\sfX)$; see \ref{almost perfect objects in the derived setting} and \ref{quasi-coherent sheaves in the formal setting}. It follows from \ref{truncation and almost of finite presentation} that $\pi_n\calO_{\calX}[n+1] \in \QCoh(\tau_{\leq n-1}\sfX)$ is almost perfect, so that the desired result now follows from \cite[8.5.1.2]{lurie2018sag} (see also \ref{a relative version of the formal GAGA}). 

We now treat the case $n=0$. In this case, \ref{completions and truncations} supplies a canonical equivalence $\tau_{\leq 0} (\sfX^\wedge_I) \simeq (\tau_{\leq 0}\sfX)^\wedge_I$, so that $\tau_{\leq 0} (\sfX^\wedge_I) \simeq \sfX^\wedge_I$: that is, $\sfX^\wedge_I$ is $0$-truncated. It follows that the vertical maps in the diagram
$$
\xymatrix{
\Map_{\DerDM}(\sfX, \sfY) \ar[r] \ar[d] & \Map_{\fDerDM}(\sfX^\wedge_I, \sfY) \ar[d] \\ 
\Map_{\DerDM}(\sfX, \tau_{\leq 0}\sfY) \ar[r] & \Map_{\fDerDM}(\sfX^\wedge_I, \tau_{\leq 0}\sfY) 
}
$$
are equivalences by virtue of \ref{truncation as right adjoint}. We note that the functor $\Xi_\mathbb{Z}: \iTop_{\CAlg^\Delta}^{\sHen} \rightarrow \iTop_{\CAlg^{\cn}_\mathbb{Z}}^{\sHen}$ of \ref{restriction to the subcategory of strictly Henselian objects and local morphisms} induces an equivalence from the $\infty$-category of $0$-truncated (formal) derived Deligne-Mumford stacks to the $\infty$-category of $0$-truncated (formal) spectral Deligne-Mumford stacks over $\mathbb{Z}$; see \ref{equivalence between 0-truncated DerDM and SpDM} and \ref{equivalence between 0-truncated fDerDM and fSpDM}. It then follows from \ref{a relative version of the formal GAGA} that the bottom horizontal map is an equivalence, so that the map $\theta_{\sfX}$ is an equivalence as desired. 
\end{proof}

\bibliography{chough_fdag}
\bibliographystyle{amsplain}

\end{document}